\documentclass[11pt]{amsart}
\usepackage{amsmath,amssymb,amsfonts,exscale}

\usepackage[curve,matrix,arrow,cmtip]{xy}
\usepackage{verbatim}

\makeatletter
\def\subsubsection{\@startsection{subsubsection}{3}%
  \z@{.5\linespacing\@plus.7\linespacing}{-.5em}%
  {\normalfont\bfseries}}
\makeatother

\makeatletter
\def\subsection{\@startsection{subsection}{2}%
  \z@{.5\linespacing\@plus.7\linespacing}{.7\linespacing}%
  {\normalfont\bfseries}}
\makeatother

\DeclareSymbolFont{bbold}{U}{bbold}{m}{n}
\DeclareSymbolFontAlphabet{\mathbbold}{bbold}

\newdir^{ (}{{}*!/-3pt/\dir^{(}}    
\newdir_{ (}{{}*!/-3pt/\dir_{(}}

\newtheorem*{theorem A}{Theorem A}
\newtheorem*{theorem B}{Theorem B}

\def\Gr{{\rm Gr}}

\def\Frob{\mathop{\rm Frob}\nolimits}

\def\Tr{\mathop{\rm Tr}\nolimits}

\def\End{\mathop{\rm End}\nolimits}

\def\Spec{\mathop{\rm Spec}\nolimits}

\def\im{\mathop{\rm im}\nolimits}

\def\Mat{\mathop{\rm Mat}\nolimits}

\def\GL{\mathop{\rm GL}\nolimits}
\def\SL{\mathop{\rm SL}\nolimits}

\def\Bun{\mathop{\rm Bun}\nolimits}
\def\Rep{\mathop{\rm Rep}\nolimits}

\def\Maps{\mathop{\rm Maps}\nolimits}

\def\VinBun{\mathop{\rm VinBun}\nolimits}

\def\Vin{\mathop{\rm Vin}\nolimits}
\def\Kost{\mathop{\rm Kostant}\nolimits}

\def\Qellbar{\mathop{\overline{\BQ}_\ell}\nolimits}
\def\IC{\mathop{\rm IC}\nolimits}

\def\Asymp{\mathop{\rm Asymp}\nolimits}

\def\add{{\rm add}}

\newbox\starbox 
\setbox\starbox=\hbox{\vphantom{\underline{\ }}\rlap{\kern4pt\small?}\smash{\lower2pt\hbox{\Large$\square$}}}

\def\hatE{{\mathchoice
  {\hbox{\rlap{\smash{\kern1pt\lower1pt\hbox{$\widehat{\phantom{\hbox{$E$}}}$}}}$E$}}
  {\hbox{\rlap{\smash{\kern1pt\lower1pt\hbox{$\widehat{\phantom{\hbox{$E$}}}$}}}$E$}}
  {\widehat E}
  {\widehat E}}}

\def\hatW{\hbox{\rlap{\smash{\lower1pt\hbox{$\widehat{\phantom{\hbox{$W$}}}$}}}$W$}}

\def\tildeW{\hbox{\rlap{\smash{\lower1pt\hbox{$\widetilde{\phantom{\hbox{$W$}}}$}}}$W$}}

\newbox\checkWbox
\setbox\checkWbox\hbox{\rlap{\smash{\kern.8pt\lower4pt\hbox{\huge \v{}}}}$W$}

\def\barBun{{\overline{\Bun}}}
\def\tildeBun{{\widetilde{\Bun}}}

\def\Gr{{\rm Gr}}
\def\D{{\rm D}}

\def\circV{{\mathchoice{\circVbig}{\circVbig}{\circVscript}{\circVscriptscript}}}
\def\circVbig{\hbox{\text{\it\r{V}}}}
\def\circVscript{\hbox{\scriptsize\text{\it\r{V}}}}
\def\circVscriptscript{\mbox{\tiny\text{\it\r{V}}}}
\def\circVprime{{\mathchoice{\circV\kern1.8pt{}^\prime}{\circV\kern1.8pt{}^\prime}
                            {\circVscript\kern1.3pt{}^\prime}{\circVscriptscript\kern1pt{}^\prime}}}

\def\circVpprime{{\mathchoice{\circV\kern1.8pt{}^{\prime \prime}}{\circV\kern1.8pt{}^{\prime \prime}}
                            {\circVscript\kern1.3pt{}^{\prime \prime}}{\circVscriptscript\kern1pt{}^{\prime \prime}}}}

\def\lambdach{\check\lambda}
\def\Lambdach{\check\Lambda}
\def\alphacheck{\check\alpha}
\def\thetacheck{\check\theta}
\def\betacheck{\check\beta}

\let\epsilon\varepsilon
\let\setminus\smallsetminus

\let\leq\leqslant
\let\geq\geqslant

\newtheorem{theorem}[subsubsection]{Theorem}

\newtheorem{corollary}[subsubsection]{Corollary}

\newtheorem{proposition-definition}[subsubsection]{Proposition-Definition}
\newtheorem{theorem-definition}[subsubsection]{Theorem-Definition}

\newtheorem{lemma}[subsubsection]{Lemma}

\newtheorem{proposition}[subsubsection]{Proposition}

\newtheorem{remark}[subsubsection]{Remark}

\newcommand{\BA}{{\mathbb{A}}}

\newcommand{\BD}{{\mathbb{D}}}

\newcommand{\BF}{{\mathbb{F}}}
\newcommand{\BG}{{\mathbb{G}}}

\newcommand{\BQ}{{\mathbb{Q}}}

\newcommand{\BZ}{{\mathbb{Z}}}

\newcommand{\Fp}{{\mathfrak{p}}}

\newcommand{\Fs}{{\mathfrak{s}}}

\newcommand{\CB}{{\mathcal B}}

\newcommand{\CH}{{\mathcal H}}
\newcommand{\CI}{{\mathcal I}}

\newcommand{\CK}{{\mathcal K}}

\newcommand{\CO}{{\mathcal O}}

\newcommand{\CU}{{\mathcal U}}

\newcommand{\CY}{{\mathcal Y}}

\newcommand{\ssec}{\subsection}
\newcommand{\sssec}{\subsubsection}

\def\longto{\longrightarrow}
\def\into{\hookrightarrow}
\let\onto\twoheadrightarrow

\def\longinto{\lhook\joinrel\longrightarrow}

\def\longonto{\ontoover{\ }}

\newbox\mybox
\def\arrover#1{\mathrel{
       \setbox\mybox=\hbox spread 1.4em
              {\hfil$\scriptstyle#1\vphantom{g}$\hfil}
       \vbox{\offinterlineskip\copy\mybox
             \hbox to\wd\mybox{\rightarrowfill}}}}
\def\larrover#1{\mathrel{
       \setbox\mybox=\hbox spread 1.4em{\hfil$\scriptstyle#1$\hfil}
       \vbox{\offinterlineskip\copy\mybox
             \hbox to\wd\mybox{\leftarrowfill}}}}

\def\ontoover#1{\mathrel{
       \setbox\mybox=\hbox spread 1.4em{\hfil$\scriptstyle#1$\hfil}
       \vbox{\offinterlineskip\copy\mybox
             \hbox to\wd\mybox{\rightarrowfill\hskip-2.8mm
                               $\rightarrow$}}}}
\def\leftontoover#1{\mathrel{
       \setbox\mybox=\hbox spread 1.4em{\hfil$\scriptstyle#1$\hfil}
       \vbox{\offinterlineskip\copy\mybox
             \hbox to\wd\mybox{$\leftarrow$\hskip-2.8mm
                               \leftarrowfill}}}}

\newbox\invlimsymbol
\setbox\invlimsymbol=\vtop{\hbox{\rm lim}\vskip-8pt
        \hbox{\hskip1pt$\scriptstyle\longleftarrow$}\vskip-1pt}

\newbox\dirlimsymbol
\setbox\dirlimsymbol=\vtop{\hbox{\rm lim}\vskip-8pt
        \hbox{\hskip1pt$\scriptstyle\longrightarrow$}\vskip-1pt}

\begin{document}

\title[Geometric Bernstein Asymptotics]{Geometric Bernstein Asymptotics and the Drinfeld-Lafforgue-Vinberg degeneration for arbitrary reductive groups}

\author{Simon Schieder}
\thanks{Dept. of Mathematics, MIT, Cambridge, MA 02139, USA}

\address{Dept. of Mathematics, MIT Cambridge, MA 02139, USA}
\email{schieder@mit.edu}

\begin{abstract}
We define and study the Drinfeld-Lafforgue-Vinberg compactification $\barBun_G$ of the moduli stack of $G$-bundles $\Bun_G$ for an arbitrary reductive group $G$; its definition is given in terms of the Vinberg semigroup of $G$, and is due to Drinfeld (unpublished). Throughout the article we prefer to view the space $\barBun_G$ as a canonical multi-parameter degeneration of $\Bun_G$ which we call the Drinfeld-Lafforgue-Vinberg degeneration $\VinBun_G$. We construct local models for the degeneration $\VinBun_G$ which ``factorize in families'' and use them to study its singularities, generalizing results of the article \cite{Sch1} which was confined with the case $G = \SL_2$.

The multi-parameter degeneration $\VinBun_G$ gives rise to, for each parabolic $P$ of $G$, a nearby cycles functor $\Psi_P$. Our main theorem expresses the stalks of these nearby cycles $\Psi_P$ in terms of the cohomology of the parabolic Zastava spaces. From this description we deduce that the nearby cycles of $\VinBun_G$ correspond, under the sheaf-function correspondence, to Bernstein's asymptotics map on the level of functions. This had been speculated by Bezrukavnikov-Kazhdan \cite{BK} and Chen-Yom Din \cite{CY} and conjectured in a precise form by Sakellaridis \cite{Sak2}.
\end{abstract}

\maketitle

\newpage

\tableofcontents

\newpage

\section{Introduction}

Let $X$ be a smooth projective curve over an algebraically closed field $k$, let $G$ be a reductive group over $k$, and let $\Bun_G$ denote the moduli stack of $G$-bundles on $X$. In this article we begin the study of the canonical relative compactification $\barBun_G$ of $\Bun_G$ due to V. Drinfeld (unpublished) for an arbitrary reductive group $G$; the case $G = \SL_2$ was studied in \cite{Sch1}. As in \cite{Sch1} we choose to work with a minor modification of the compactification $\barBun_G$ which we denote by $\VinBun_G$ and refer to as the \textit{Drinfeld-Lafforgue-Vinberg degeneration} of $\Bun_G$.

\medskip

For $G = \GL_n$ certain smooth open substacks of the space $\barBun_G$ were used by Drinfeld and by L. Lafforgue in their celebrated work on the Langlands correspondence for function fields (\cite{Varieties of modules of F -sheaves}, \cite{Cohomology of compactified manifolds of modules of F -sheaves of rank 2}, \cite{Lafforgue JAMS}). The spaces $\barBun_G$ and $\VinBun_G$ are however already singular for $G=\SL_2$. The goal of the present article is to begin the study of their singularities for an arbitrary reductive group $G$, generalizing our earlier work \cite{Sch1} for $G=\SL_2$. This study is originally motivated by the geometric Langlands program (\cite{G3}, \cite{G4}), for example by applications to Drinfeld's and Gaitsgory's \textit{miraculous duality} (\cite{DrG1}, \cite{DrG2}, \cite{G2}); see Subsection \ref{stalks and applications} below for such applications of the current work. In the present article we however focus on a novel application to the classical theory, to the \textit{Bernstein asymptotics map} on the level of functions.

\bigskip

\ssec{The degeneration $\VinBun_G$ for arbitrary $G$}

\sssec{The Vinberg semigroup}
In \cite{V} E. B. Vinberg has constructed a canonical multi-parameter degeneration $\Vin_G \to \BA^r$ of an arbitrary reductive group $G$ of semisimple rank $r$; this degeneration carries a semigroup structure and is called the \textit{Vinberg semigroup}. Its fibers over the complement of the union of all coordinate hyperplanes are isomorphic to the group $G$; its fibers over the coordinate hyperplanes can be described in group-theoretic terms related to the parabolic subgroups of $G$. A certain well-behaved open subvariety of the Vinberg semigroup, the \textit{non-degenerate locus}, is closely related to the wonderful compactification of the adjoint group $G_{adj}$ of $G$ constructed by De Concini and Procesi \cite{DCP}.

\medskip

\sssec{The definition of $\VinBun_G$}
As the Vinberg semigroup $\Vin_G$ carries a natural $G \times G$-action, one may form the mapping stack
$$\Maps(X, \Vin_G / G \times G)$$
parametrizing maps from the curve $X$ to the quotient $\Vin_G / G \times G$. The stack $\VinBun_G$ is then obtained from this mapping stack by imposing certain non-degeneracy conditions. Like the Vinberg semigroup $\Vin_G$, the stack $\VinBun_G$ comes equipped with a natural map
$$\VinBun_G \ \longto \ \BA^r \, ;$$
just like $\Vin_G$ forms a canonical degeneration of the group $G$, the stack $\VinBun_G$ forms, via this map, a canonical degeneration of $\Bun_G$. The compactification $\barBun_G$ can be obtained from $\VinBun_G$ as the quotient by a maximal torus of $G$.

\medskip

\sssec{The case $G = \SL_2$}
As is discussed in \cite{Sch1}, for $G = \SL_2$ the degeneration $\VinBun_G$ can be described very concretely as follows: It parametrizes triples $(E_1, E_2, \varphi)$ consisting of two $\SL_2$-bundles $E_1$, $E_2$ on the curve $X$ together with a morphism of the associated vector bundles $\varphi: E_1 \to E_2$ which is required to be not the zero map. Taking the determinant of the map $\varphi$ yields the desired map
$$\VinBun_G \ \longto \ \BA^1 \, .$$

\medskip

\ssec{Stratifications}

\sssec{Stratification by parabolics}
Let $T$ denote a maximal torus of the reductive group $G$, let $B$ denote a Borel subgroup containing $T$ and let $Z_G$ denote the center of $G$. The target affine space $\BA^r$ of the map $\Vin_G \to \BA^r$ naturally forms a semigroup completion of the adjoint torus $T/Z_G$. Thus its coordinate stratification is naturally indexed by standard parabolic subgroups $P$ of $G$. This stratification induces stratifications of $\Vin_G$ and of $\VinBun_G$ which are also indexed by standard parabolic subgroups:
$$\VinBun_G \ \ = \ \ \bigcup_{P} \ \VinBun_{G,P}$$

\medskip

\sssec{Defect stratifications}
To each point in any of the loci $\VinBun_{G,P}$ we associate a simpler geometric datum which we call its \textit{defect value}; the defect value governs the singularity of the point in the moduli space $\VinBun_G$. For $G=\SL_2$, the defect values are effective divisors on the curve $X$. For an arbitrary reductive group $G$ and for $P=B$, the defect values are effective divisors on $X$ valued in the monoid of positive coweights $\Lambdach_G^{pos}$ of $G$. For an arbitrary reductive group $G$ and an arbitrary parabolic $P$, the defect values are certain points in the affine Grassmannian $\Gr_M$ of the Levi $M$ of $P$.
We obtain finer stratifications, the \textit{defect stratifications}, of the loci $\VinBun_{G,P}$ by requiring certain numerical invariants of the defect value to remain constant.

\medskip

\ssec{Main results -- Geometry}

\sssec{Stalks of nearby cycles}

The degeneration $\VinBun_G \to \BA^r$ gives rise to, for each standard parabolic $P$ of $G$, a one-parameter family connecting the $G$-stratum $\VinBun_{G,G}$ and the $P$-stratum $\VinBun_{G,P}$. Let $\Psi_P$ denote the nearby cycles perverse sheaf of the one-parameter family corresponding to the parabolic $P$. Vaguely speaking, our main theorem regarding the geometry of $\VinBun_G$ then states:

\begin{theorem A}
The stalks of $\Psi_P$ along the defect stratification of $\VinBun_{G,P}$ are isomorphic to the cohomology of the defect-free parabolic Zastava spaces from \cite{BFGM}.
\end{theorem A}

We refer the reader to Sections \ref{Statements of theorems -- Geometry} and \ref{Proofs I --- Construction of local models for VinBun_G} for the definitions of the objects appearing in the theorem, and to Theorem \ref{nearby cycles theorem} for a precise formulation. As will be discussed below, Theorem A shows that the nearby cycles of the degeneration $\VinBun_G$ may be regarded as a global geometric version of the Bernstein asymptotics map. In this context, the description of the stalks of the nearby cycles in Theorem A in terms of the cohomology of the Zastava spaces may be viewed as a geometric analog of the classical Gindikin-Karpelevich formulas for the Bernstein asymptotics (see Subsection \ref{Proof of Sakellaridis's conjecture} below).

\medskip

\sssec{Stalks of the extension of the constant sheaf, and applications in the geometric Langlands program}
\label{stalks and applications}

\mbox{}

\medskip

We also give a description of the stalks of the $*$-extension of the constant sheaf from the open stratum $\VinBun_{G,G}$ which is closely related to Theorem A; see Theorem \ref{ij} for its formulation.
This description of the $*$-extension of the constant sheaf provides the geometric input for Gaitsgory's proof that the \textit{miraculous duality} (\cite{DrG1}, \cite{DrG2}, \cite{G2}) acts as the identity on \textit{cuspidal objects}.
Furthermore, this description will be applied in the forthcoming PhD thesis of Wang \cite{W2} to geometrically construct Drinfeld's \textit{strange invariant bilinear form} on the space of automorphic forms for arbitrary reductive groups.

\bigskip

\ssec{Main results -- Bernstein asymptotics}

For this paragraph only let $G$ now denote a reductive group over a non-archimedean local field $F$. Let $N$ denote the unipotent radical of the Borel $B$ of $G$. The \textit{Bernstein asymptotics map} is a map of $G \times G$-modules
$$\Asymp: \ C^{\infty}(G) \ \longto \ C^{\infty}((G/N \times G/N^-)/T) \, .$$
It can be defined either via a universal property related to the asymptotics of matrix coefficients,
or as a composition of the \textit{orispheric transform} with the inverse of the \textit{intertwining operator} (see \cite{BK}, \cite{SakV}, \cite{Sak1}). Our geometric results imply that the nearby cycles of $\VinBun_G$ form a geometric or categorical version of the Bernstein asymptotics map, as we now discuss.

\medskip

Before making a more precise statement, we recall that Bezrukavnikov and Kazhdan \cite{BK} have used the Bernstein asymptotics map to prove Bernstein's second adjointness theorem for reductive groups over non-archimedean local fields. In \cite{BK} they speculate whether the Bernstein asymptotics map is related to some nearby cycles construction on the geometric level. More precisely, they observe that their description of the Bernstein map as the composition of the orispheric transform with the inverse of the intertwining operator is formally analogous to the definition of the \textit{twisted Harish-Chandra functor} in \cite{BFO} (see also \cite{ENV}, \cite{CY}), which works in the finite-dimensional setting (i.e., over an algebraically closed field instead of over a local field). Since it is shown in \cite{BFO} that the twisted Harish-Chandra functor can be realized as the functor of Verdier specialization in the De Concini-Procesi wonderful compactification, they ask whether similarly the Bernstein asymptotics can be viewed as some nearby cycles procedure of an appropriate space; the same question has also been raised by Chen and Yom Din \cite{CY}.

\medskip

Similar predictions have been made by Sakellaridis and Venkatesh; in \cite{SakV} they have constructed asymptotics maps for arbitrary spherical varieties over non-archimedean local fields, which reduce to the above case of Bernstein asymptotics when the spherical variety is the group itself. Sakellaridis \cite{Sak2} has given a precise conjecture along the lines of the question of Bezrukavnikov and Kazhdan, relating the Bernstein asymptotics to the nearby cycles of the degeneration $\VinBun_G$. We deduce this conjecture for arbitrary reductive groups from Theorem A. We refer the reader to Theorem \ref{theorem Sakellaridis's conjecture} below for a precise statement; broadly speaking, its assertion is:

\medskip

\begin{theorem B}[Sakellaridis's conjecture from \cite{Sak2}]
The nearby cycles sheaf $\Psi_{\VinBun_G}$ \textit{factorizes}, i.e., its stalks decompose into tensor products of \textit{local factors}. The functions corresponding to the local factors of $\Psi_{\VinBun_G}$ under the sheaf-function correspondence agree with the Bernstein asymptotics of the basic Schwartz functions.
\end{theorem B}

\bigskip

\ssec{Proofs via local models}

We study the degeneration $\VinBun_G$ by constructing certain \textit{local models} for it which feature the same singularities as $\VinBun_G$ but possess a \textit{factorization property}, in the sense of Beilinson and Drinfeld (\cite{BD1}, \cite{BD2}). Our models thus play an analogous role for the space $\VinBun_G$ as the Zastava spaces from \cite{FM}, \cite{FFKM}, \cite{BFGM} play for Drinfeld's spaces of quasimaps (see e.g. \cite{BG1}, \cite{BG2}). Our local models may in fact also be viewed as canonical degenerations of the Zastava spaces.

\medskip

We will in fact construct one local model for each parabolic $P$ of $G$, which will then be used to study the singularities of the degeneration into the $P$-locus $\VinBun_{G,P}$ of $\VinBun_G$. We furthermore point out that our local models are not quite factorizable in the sense of Beilinson and Drinfeld, but rather \textit{factorizable in families}: They themselves form multi-parameter degenerations whose fibers are factorizable in compatible ways. Our geometric main theorems are then deduced from certain geometric properties of the local models; Theorem B follows from Theorem A under the sheaf-function correspondence by computing the function corresponding to the cohomology sheaves of the defect-free parabolic Zastava spaces.

\bigskip

\ssec{Structure of the article}

We now briefly outline the contents of the individual sections.
In Section~\ref{The compactification and the degeneration} we recall various facts about the Vinberg semigroup, define the spaces $\VinBun_G$ and $\barBun_G$, and discuss their basic properties. In Section \ref{The defect stratification} we construct the aforementioned defect stratifications of the loci $\VinBun_{G,P}$. In Sections \ref{Statements of theorems -- Geometry} and \ref{Statements of theorems -- Bernstein asymptotics} we state our main theorems, including precise versions of Theorems A and B sketched in this introduction.

\medskip

The remaining sections deal with the proofs of the above theorems. In Section \ref{Proofs I --- Construction of local models for VinBun_G} we construct the local models for the loci $\VinBun_{G,P}$ and study their geometry. In Section \ref{Proofs II --- Sheaves} we deduce the aforementioned results about the nearby cycles and the $*$-extension of the constant sheaf from the geometric facts of the previous section. In Section \ref{Proofs III -- Bernstein asymptotics} we compute the function corresponding to the nearby cycles under the sheaf-function dictionary and deduce the results about the Bernstein asymptotics map.

\bigskip

\ssec{Conventions and notation}
\label{Conventions and notation}

We will invoke a formalism of mixed sheaves; for concreteness we will work with the formalism of $\ell$-adic Weil sheaves: We assume the curve $X$ is defined over a finite field, and work with Weil sheaves over the algebraic closure $k$ of the finite field. For a scheme or stack $Y$, we denote by $D(Y)$ the derived category of constructible $\Qellbar$-sheaves on $Y$. We fix once and for all a square root $\Qellbar(\tfrac{1}{2})$ of the Tate twist $\Qellbar(1)$. We normalize all IC-sheaves to be pure of weight $0$; thus on a smooth variety $Y$ the IC-sheaf is equal to $\Qellbar[\dim Y](\tfrac{1}{2} \dim Y)$. Our conventions for nearby cycles are stated in Subsection \ref{Recollections} below. We denote the exterior product of sheaves on a product space by the symbol $\boxtimes$. In the case of a fiber product over a space $Y$ we denote by $\boxtimes_{Y}$ the $*$-restriction of the exterior product to the fiber product over $Y$, shifted by $[- \dim Y]$ and twisted by $(-\tfrac{1}{2} \dim Y)$. Finally, we denote the restriction of a space or a sheaf to a ``disjoint locus'' by the symbol $\circ$, whenever there is no confusion about what the disjointness is referring to. For example, we denote by
$$X^{(n_1)} \stackrel{\circ}{\times} X^{(n_2)}$$
the open subset of the product $X^{(n_1)} \times X^{(n_2)}$ of symmetric powers of the curve $X$ consisting of those pairs of effective divisors with disjoint supports, and call it {the disjoint locus of} $X^{(n_1)} \times X^{(n_2)}$.

\bigskip
\bigskip
\bigskip

\ssec{Acknowledgements}

I would like to express my sincere gratitude to Dennis Gaitsgory and Vladimir Drinfeld, for recommending to study the space $\barBun_G$, for many conversations and suggestions, as well as for their continued guidance and support. I would like to thank Yiannis Sakellaridis for drawing my attention to the Bernstein asymptotics map as well as for numerous conversations. I would like to thank Jonathan Wang for conversations about the Vinberg semigroup $\Vin_G$ and about the reductive monoid $\overline{M}$. Finally, I would like to thank Roman Bezrukavnikov for helpful conversations related to his work \cite{BK}.

\newpage

\section{The compactification and the degeneration}
\label{The compactification and the degeneration}

\bigskip

\ssec{The Vinberg semigroup}

E. B. Vinberg has associated to any reductive group $G$ a canonical algebraic semigroup, the \textit{Vinberg semigroup} $\Vin_G$ of $G$ (\cite{V}). Vinberg's work assumes the characteristic of the base field to be $0$; the case of arbitrary characteristic can be found in \cite{Ri1}, \cite{Ri2}, \cite{Ri3}, \cite{Ri4}, and \cite{BKu}.
Here we recall the definition of $\Vin_G$ and some of its basic properties. For proofs and further background about reductive semigroups and the Vinberg semigroup we refer the reader to the above articles as well as to \cite{Pu}, \cite{Re}, and \cite{DrG2}.

\medskip

\sssec{Notation related to the group}
Let $G$ be a reductive group over $k$, let $r$ denote the semisimple rank of $G$, and let $Z_G$ denote the center of $G$.
For simplicity we assume that the derived group $[G,G]$ of $G$ is simply connected.
We fix a maximal torus $T$ of $G$ and a Borel subgroup $B$ containing $T$, and denote by $W$ the Weyl group of $G$ and by $w_0$ its longest element. Let $\Lambda_G$ denote the weight lattice of $G$, let $\Lambdach_G$ denote the coweight lattice of $G$, let $\CI$ denote the set of vertices of the Dynkin diagram of $G$, let $(\alpha_i)_{i \in \CI} \in \Lambda_G$ denote the simple roots, and let $(\alphacheck_i)_{i \in \CI} \in \Lambdach_G$ denote the simple coroots. We denote by $\Lambda_G^+$ the collection of dominant weights, and by $\Lambda_G^{pos}$ the collection of positive weights, and analogously for $\Lambdach_G$. We denote by $\leq$ the usual partial order on $\Lambda_G$ and $\Lambdach_G$.

\medskip

\sssec{Notation related to a parabolic}
By a parabolic we will by default mean a \textit{standard} parabolic, i.e., a parabolic containing the chosen Borel $B$. For a parabolic $P$ we denote by $U_P$ its unipotent radical and by $M$ the corresponding Levi quotient and subgroup. The subset of vertices in $\CI$ corresponding to the parabolic $P$ will be denoted by $\CI_M$, the semisimple rank of $M$ by $r_M$, and its center by $Z_M$. Finally, we denote by $\Lambdach_{G,P}$ the quotient
$$\Lambdach_{G,P} \ := \ \Lambdach_G / \sum_{i \in \CI_M} \BZ \check\alpha_i$$
and by $\Lambdach_{G,P}^{pos}$ the image of $\Lambdach_G^{pos}$ under the natural projection $\Lambdach_G \onto \Lambdach_{G,P}$. Using the monoid $\Lambdach_{G,P}^{pos}$ we define a partial ordering $\leq$ on $\Lambdach_{G,P}$ as for $\Lambdach_G$.

\medskip

\sssec{The enhanced group}
We define the \textit{enhanced group} of $G$ as
$$G_{enh} \ = \ (G \times T) /Z_{G}$$
where the center $Z_G$ of $G$ acts anti-diagonally on $G \times T$, i.e., by the formula $(g,t).z = (zg, z^{-1}t)$. The group $G$ is naturally a subgroup of $G_{enh}$ via the inclusion of the first coordinate
$$G \ \longinto \ G_{enh} \, .$$

\medskip

\sssec{The definition of $\Vin_G$ via classification of reductive monoids}
\label{The definition of Vin_G via classification of reductive monoids}
The Vinberg semigroup $\Vin_G$ is an affine algebraic monoid whose group of units is open and dense in $\Vin_G$ and equal to the reductive group $G_{enh}$. We now recall its definition via the Tannakian formalism and the classification of \textit{reductive monoids}, i.e., the classification of irreducible affine algebraic monoids whose group of units is dense, open, and a reductive group.

\medskip

Let $\Rep(G_{enh})$ denote the category of finite-dimensional representations of $G_{enh}$.
By the classification of reductive monoids (see \cite{Pu}, \cite{Re}, \cite{V}), the monoid $\Vin_G$ is uniquely determined by the full subcategory
$$\Rep(\Vin_G) \ \subset \ \Rep(G_{enh})$$
consisting of all those representations $V \in \Rep(G_{enh})$ for which the action of $G_{enh}$ extends to an action of the monoid $\Vin_G$.
We can thus define $\Vin_G$ by specifying this full subcategory $\Rep(\Vin_G)$ of $\Rep(G_{enh})$. To do so, note first that any representation $V$ of $G_{enh}$ admits a canonical decomposition as $G_{enh}$-representations
$$V \ = \ \bigoplus_{\lambda \in \Lambda_T} V_\lambda$$
according to the action of the center $Z_{G_{enh}} = (Z_G \times T)/Z_G = T$, i.e., such that $Z_{G_{enh}} = T$ acts on each $V_\lambda$ by the character $\lambda$. Each $V_\lambda$ also naturally forms a $G$-representation via the inclusion $G \into G_{enh}$; its central character as a $G$-representation is equal to the restriction $\lambda|_{Z_G}$.

\medskip

With this notation, the subcategory $\Rep(\Vin_G)$ of $\Rep(G_{enh})$ is defined as follows: It contains a representation $V \in \Rep(G_{enh})$ if and only if for each $\lambda \in \Lambda_T$ the weights of the summand $V_{\lambda}$, considered as a representation of $G$, are all $\leq \lambda$.

\medskip

\sssec{Basic properties of the Vinberg semigroup}
The variety $\Vin_G$ is normal and carries a natural $G \times G$-action which extends the natural $G \times G$-action on $G_{enh}$. It moreover carries a natural $T$-action which extends the $T$-action on $G_{enh} = (G \times T) / Z_G$ defined by acting on the second factor; this action commutes with the $G \times G$-action, and will simply be referred to as \textit{the} $T$-action on $\Vin_G$.

\medskip

The Vinberg semigroup can be viewed as the total space of a canonical multi-parameter degeneration of the group $G$, as we recall next. To do so, consider the adjoint torus $T_{adj} = T/Z_G$ and recall that the collection of simple roots $(\alpha_i)_{i \in \CI}$ of $G$ yields a canonical identification
$$T_{adj} \ \stackrel{\cong}{\longto} \ \BG_m^r \, .$$
In other words, the simple roots form canonical affine coordinates on $T_{adj}$. Allowing these coordinates to vanish we obtain a
canonical semigroup completion $T_{adj}^+$ of $T_{adj}$ by defining
$$T_{adj}^+ \ := \ \BA^r \ \supset \ \BG_m^r \ = T_{adj} \, ,$$
where the structure of algebraic semigroup on $\BA^r$ is given by component-wise multiplication. The natural action of $T$ on $T_{adj}$ extends to an action of $T$ on $T_{adj}^+$.

\medskip

The semigroup $\Vin_G$ admits a natural flat semigroup homomorphism
$$v: \ \Vin_G \ \longto \ T_{adj}^+ = \BA^r$$
extending the natural projection map $G_{enh} \longto T_{adj}$. The map $v$ is $G \times G$-invariant and $T$-equivariant for the above $T$-actions on $\Vin_G$ and on $T_{adj}^+$. The fiber of the map $v$ over the point $1 \in T_{adj}^+$ is canonically identified with the group $G$. It is in this sense that the Vinberg semigroup is a multi-parameter degeneration of the group $G$. In Subsection \ref{The stratification parametrized by parabolics} below we will recall descriptions of all other fibers of the map $v$ in group-theoretic terms.

\medskip

\sssec{The canonical section}
\label{The canonical section}
Recall that we have fixed choices of a maximal torus and a Borel subgroup $T \subset B \subset G$. These choices give rise to a section
$$\Fs: \ T_{adj}^+ \ \longto \ \Vin_G$$
of the map
$$v: \ \Vin_G \ \longto \ T_{adj}^+ \, .$$
The section $\Fs$ is uniquely characterized as follows. Note first that the map
$$T \ \longto \ G \times T \, , \ \ t \ \longmapsto \ (t^{-1}, t)$$
descends to a map $T_{adj} \longto G_{enh}$, and that the latter map forms a section of the map $G_{enh} \longto T_{adj}$. One can then show that this section extends to the desired section $\Fs$ of the map $v$, and that the image under $\Fs$ of any point in $T_{adj}^+$ in fact lies in the open $G \times G$-orbit of the corresponding fiber of $v$.
This shows that the section $\Fs$ in fact factors through the \textit{non-degenerate locus} $\sideset{_0}{_G}\Vin$ of $\Vin_G$, which we recall next.

\sssec{The non-degenerate locus}
We now recall a natural dense open subvariety
$$\sideset{_0}{_G}\Vin \ \ \subset \ \ \Vin_G$$
which we will refer to as the \textit{non-degenerate locus} of $\Vin_G$. It is uniquely characterized by the fact that it meets each fiber of the map $v: \Vin_G \to T_{adj}^+$ in the open $G \times G$-orbit of that fiber; i.e., for any $t \in T_{adj}^+$ we have:
$$\Vin_G|_{t} \, \cap \sideset{_0}{_G}\Vin \ \ = \ \ G \cdot \Fs(t) \cdot G$$
For a Tannakian characterization of $\sideset{_0}{_G}\Vin$ we refer the reader to \cite[Sec. D4]{DrG2}. The open subvariety $\sideset{_0}{_G}\Vin$ of $\Vin_G$ is not only $G \times G$-stable but also $T$-stable, and the restriction of the map $v$ to $\sideset{_0}{_G}\Vin$ is smooth.

\sssec{The stratification parametrized by parabolics}
\label{The stratification parametrized by parabolics}
Consider the coordinate stratification of the completed adjoint torus $T_{adj}^+ = \BA^r$. Its strata are stable under the action of $T$ and are naturally indexed by subsets of the Dynkin diagram $\CI$ of $G$, or equivalently by standard parabolic subgroups of $G$:
$$T_{adj}^+ \ \ = \ \ \bigcup_{P} \ T^+_{adj, P} \ .$$
Each stratum $T^+_{adj, P}$ contains a canonical point $c_P$ which is defined as follows. Let $M$ denote the Levi quotient of the parabolic $P$, and let $\CI_M \subset \CI$ denote the subset of $\CI$ consisting of those vertices corresponding to $P$. Then in the coordinates $T_{adj}^+ = \BA^r$ we define $(c_P)_i = 1$ for $i \in \CI_M$ and $(c_P)_i = 0$ for $i \notin \CI_M$. Thus for example $c_G = 1 \in T_{adj}$ and $c_B = 0 \in T_{adj}^+$.

\medskip

Pulling back the stratification of $T_{adj}^+$ along the map $v$ we obtain a stratification of $\Vin_G$ indexed by standard parabolic subgroups of $G$:
$$\Vin_G \ \ = \ \ \bigcup_{P} \ \Vin_{G,P} \ .$$
Note that
$$\Vin_{G,G} \ = \ G_{enh} \ = \ (G \times T) / Z_G \ = \ G \times T_{adj}$$
as varieties over $T_{adj}$, where the last identification is induced by the map
$$(g,t) \ \mapsto \ (g t^{-1}, t) \, .$$
Below we recall the description of the strata $\Vin_{G,P}$ in terms of the group $G$. Note first that since the $T$-action on $T_{adj}^+$ is transitive when restricted to any of the strata $T_{adj, P}^+$, all fibers of the $T$-equivariant map $\Vin_{G,P} \to T_{adj,P}^+$ are isomorphic. In fact, using the section $\Fs$ from Subsection \ref{The canonical section} one obtains an action of $T_{adj}$ on $\Vin_G$ which by construction lifts the action of $T_{adj}$ on $T_{adj}^+$. This implies the following stronger assertion:

\medskip

\begin{remark}
The fiber bundle $\Vin_{G,P} \to T_{adj,P}^+$ is trivial.
\end{remark}

\medskip

We will thus confine ourselves to describing the fiber $\Vin_G|_{c_P}$ of $\Vin_G$ over the point $c_P \in T_{adj,P}^+$. To do so, recall first that a scheme Z over $k$ is called {\it strongly quasi-affine} if its ring of global functions $\Gamma(Z, \CO_Z)$ is a finitely generated $k$-algebra and if the natural map
$$Z \ \longto \ \overline{Z} \ := \ \Spec (\Gamma(Z, \CO_Z))$$
is an open immersion. If $Z$ is strongly quasi-affine we will call $\overline{Z}$ its {\it affine closure}. We first recall:

\medskip

\begin{lemma}
Let the Levi quotient $M$ of a parabolic $P$ act diagonally on the right on the product $G/U_P \times G/U_{P^-}$. Then the quotient
$$(G/U_P \times G/U_{P^-})/M$$
is strongly quasi-affine.
\end{lemma}

\medskip

Denoting by $\overline{(G/U_P \times G/U_{P^-})/M}$ the affine closure of $(G/U_P \times G/U_{P^-})/M$, we now recall:

\medskip

\begin{lemma}
\label{Vinberg strata}
The $G \times G$-action on the point $\Fs(c_P) \in \sideset{_0}{_G}\Vin|_{c_P}$ induces an isomorphism
$$(G/U_P \times G/U_{P^-})/M \ \ \ \stackrel{\cong}{\longto} \ \ \ \sideset{_0}{_G}\Vin|_{c_P} \, ,$$
which in turn induces an isomorphism
$$\overline{(G/U_P \times G/U_{P^-})/M} \ \ \ \stackrel{\cong}{\longto} \ \ \ \Vin_G|_{c_P}$$
on the affine closure.
In particular, the latter isomorphism is $G \times G$-equivariant for the natural $G \times G$-actions. Taking $P=B$ we find that
$$\Vin_{G,B} \ \ \ \cong \ \ \ \overline{(G/N \times G/N^-)/T} \, .$$
\end{lemma}

\bigskip

While it will be essential for us to consider the entire Vinberg semigroup $\Vin_G$, we remark that the non-degenerate locus $\sideset{_0}{_G}\Vin$ is closely related to the wonderful compactification of De Concini and Procesi: Let $G_{adj} = G/Z_G$ denote the adjoint group of $G$, and let $\overline{G_{adj}}^{DCP}$ denote its wonderful compactification (see \cite{DCP}, \cite{BKu}). Then we have:

\medskip

\begin{remark}
The $T$-action on $\sideset{_0}{_G}\Vin$ is free and induces an isomorphism
$$\sideset{_0}{_G}\Vin/T \ \ \cong \ \ \overline{G_{adj}}^{DCP} \, .$$
\end{remark}

\medskip

\sssec{The Vinberg semigroup for $G = \SL_2$}
\label{The Vinberg semigroup for G = SL_2}

As an illustration we now discuss the above notions in the case $G=\SL_2$; this case has implicitly been used in the work \cite{Sch1}, which was concerned with the study of the Drinfeld-Lafforgue-Vinberg degeneration in the case $G=\SL_2$. For $G = \SL_2$ the Vinberg semigroup is equal to the semigroup of $2 \times 2$ matrices~$\Mat_{2 \times 2}$. The $\SL_2 \times \SL_2$-action is given by left and right multiplication, and the action of $T = \BG_m$ by scalar multiplication. The semigroup homomorphism $v$ is equal to the determinant map
$$v: \ \Vin_G = \Mat_{2 \times 2} \ \stackrel{\det}{\longto} \ \BA^1 = T_{adj}^+ \, .$$
The canonical section $\Fs$ takes the form
$$\Fs: \ \BA^1 \ \longto \ \Mat_{2 \times 2} \, , \ \ c \ \longmapsto \bigl( \begin{smallmatrix} 1 & 0 \\ 0 & c \end{smallmatrix} \bigr) \, .$$
For $G = \SL_2$ the Vinberg semigroup possesses only two strata: The $G$-locus
$$\Vin_{G,G} \ = \ v^{-1}(\BA^1 \setminus \{ 0\}) \ \cong \ \GL_2 \, ,$$
and the $B$-locus
$\Vin_{G,B} = v^{-1}(0)$ consisting of all singular $2 \times 2$ matrices. Finally, the non-degenerate locus $\sideset{_0}{_G}\Vin \subset \Vin_G$
is equal to the subset
$$\Mat_{2 \times 2} \setminus \{0\} \ \ \subset \ \ \Mat_{2 \times 2}$$
of non-zero matrices.

\bigskip

\ssec{Definition of $\VinBun_G$ and of $\barBun_G$}

We now give the definition of the compactification $\barBun_G$ for an arbitrary reductive group $G$; this definition is due to Drinfeld (unpublished). In fact, we first give the definition of the degeneration $\VinBun_G$, and then define $\barBun_G$ as a torus quotient of $\VinBun_G$.

\medskip

\sssec{Notation}
Let $G$ be a reductive group over $k$ and let $X$ be a smooth projective curve over $k$.
Recall that for a stack $\CY$ the sheaf of groupoids $\Maps(X,\CY)$ parametrizing maps from the curve $X$ to the stack $\CY$ is defined as
$$\Maps(X, \CY)(S) \ = \ \CY(X \times S) \, .$$
Thus for example we have $\Bun_G = \Maps(X, \cdot/G)$.
Similarly, given an open substack $\overset{\circ}{\CY} \subset \CY$, the sheaf of groupoids
$$\Maps_{gen}(X, \CY \supset \overset{\circ}{\CY})$$
associates to a scheme $S$ the full sub-groupoid of $\Maps (X, \CY)(S)$ consisting of those maps $X \times S \to \CY$ satisfying the following condition: We require that for every geometric point $\bar s \to S$ there exists an open dense subset of $X \times \bar{s}$ on which the restricted map $X \times \bar s \to \CY$ factors through the open substack~$\overset{\circ}{\CY} \subset \CY$. 

\medskip

\sssec{Definition of $\VinBun_G$}
Quotienting out by the $G \times G$-action on $\Vin_G$ we obtain an open substack
$$\sideset{_0}{_G}\Vin / G \times G \ \ \ \subset \ \ \ \Vin_G / G \times G \, .$$
We then define the \textit{Drinfeld-Lafforgue-Vinberg degeneration} $\VinBun_G$ for an arbitrary reductive group $G$ as
$$\VinBun_G \ \ := \ \ \Maps_{gen} \, (X, \ \text{Vin}_G / G \times G \ \supset \ \sideset{_0}{_G}\Vin / G \times G) \, .$$
Since the curve $X$ is proper, the map $v: \Vin_G \longto T_{adj}^+$ induces a map
$$v: \ \VinBun_G \ \longto \ T_{adj}^+  = \BA^r \, .$$
The map $v$ makes $\VinBun_G$ into a multi-parameter degeneration of $\Bun_G$ in the sense that any fiber of the map $v$ over a point in $T_{adj} \subset T_{adj}^+$ is isomorphic to $\Bun_G$.

\medskip

\sssec{The space $\VinBun_G$ for $G = \SL_2$}
Using the description of $\Vin_G$ for $G=\SL_2$ in Subsection \ref{The Vinberg semigroup for G = SL_2} above one recovers the concrete definition of $\VinBun_G$ for $G=\SL_2$ given in \cite{Sch1}: For $G=\SL_2$ an $S$-point of $\VinBun_G$ consists of the data of two vector bundles $E_1$, $E_2$ of rank $2$ on $X \times S$, together with trivializations of their determinant line bundles $\det E_1$ and $\det E_2$, and a map of coherent sheaves
$$\varphi: \ E_1 \ \longto \ E_2 \, ,$$
satisfying the condition that for each geometric point $\bar{s} \to S$ the map
$$\varphi|_{X \times \bar{s}} : \ \ E_1|_{X \times \bar{s}} \ \longto \ E_2|_{X \times \bar{s}}$$
is not the zero map; in other words, the map $\varphi|_{X \times \bar{s}}$ is required to not vanish generically on the curve $X \times \bar{s}$. The map $v: \VinBun_G \longto \BA^1$ is obtained by sending the above data to the point $\det(\varphi) \in \BA^1(S)$.

\medskip

\sssec{Definition of $\barBun_G$}
\label{Definition of barBun_G}

Since the action of $T$ on $\Vin_G$ commutes with the $G \times G$-action, it induces an action of $T$ on $\VinBun_G$; by construction the map $v: \VinBun_G \to T_{adj}^+$ is $T$-equivariant. We then define $\barBun_G$ as the quotient by this action:
$$\barBun_G \ \ := \ \ \VinBun_G/T$$
In other words, we define $\barBun_G$ as the fiber product
$$\xymatrix@+10pt{
\barBun_G \ar[rr] \ar[d] & & \Maps_{gen} \, (X, \ \text{Vin}_G / G \times G \times T \, \supset \, \sideset{_0}{_G}\Vin / G \times G \times T) \ar[d]  \\
T_{adj}^+/T \ar[rr] & & \Maps(X, T_{adj}^+/T) \\
}$$
where the bottom map assigns to a point the corresponding constant map, and the right map is induced by the map $\Vin_G \to T_{adj}^+$. Hence the spaces $\barBun_G$ and $\VinBun_G$ fit into a cartesian square
$$\xymatrix@+10pt{
\VinBun_G \ar[r] \ar^{v}[d] & \barBun_G \ar^{\bar{v}}[d]  \\
T_{adj}^+ \ar[r] & T_{adj}^+/T \\
}$$
where the horizontal arrows are $T$-bundles. In particular, the study of the singularities of the space $\barBun_G$ and the map $\bar{v}$ are equivalent to the study of the singularities of the space $\VinBun_G$ and the map $v$. Thus for the remainder of the article we will be mainly concerned with the degeneration $\VinBun_G$.

\medskip

\sssec{The $T_{adj}$-action on $\VinBun_G$}
\label{The T_adj-action on VinBun_G}
By its definition as a reductive monoid with unit group $G_{enh}$, the Vinberg semigroup $\Vin_G$ carries a natural $G_{enh} \times G_{enh}$-action. Since $G \times G$ forms a normal subgroup in $G_{enh} \times G_{enh}$ with quotient $T_{adj} \times T_{adj}$, the quotient $\Vin_G / G \times G$ carries a natural $T_{adj} \times T_{adj}$-action. By construction the $T_{adj}$-action of each of the two factors individually makes the map
$$\Vin_G / G \times G \ \longto \ T_{adj}^+$$
equivariant with respect to the natural $T_{adj}$-action on $T_{adj}^+$. In particular we record:

\begin{remark}
\label{T_adj-action on VinBun_G}
The stack $\VinBun_G$ carries a natural $T_{adj} \times T_{adj}$-action such that the $T_{adj}$-action of either of the two factors makes the map
$$\VinBun_G \ \longto \ T_{adj}^+$$
equivariant with respect to the natural $T_{adj}$-action on $T_{adj}^+$.
\end{remark}

\sssec{Stratification by parabolics}
\label{Stratification by parabolics}
The stratification of $T_{adj}^+$ indexed by parabolic subgroups $P$ of $G$ induces, via pullback along the map $v$, a stratification
$$\VinBun_G \ \ = \ \ \bigcup_{P} \ \VinBun_{G,P} \, ,$$
and similarly for $\barBun_G$.

\medskip

We will introduce stratifications of the loci $\VinBun_{G,P}$ in Section \ref{The defect stratification} below. As for $\Vin_G$ we note that since the $T$-action on $T_{adj,P}^+$ is transitive, the fibers of the map $\VinBun_{G,P} \to T_{adj,P}^+$ are all isomorphic, and we may restrict our attention to the fiber $\VinBun_G|_{c_P}$. In fact, Remark \ref{T_adj-action on VinBun_G} implies that the following stronger assertion holds:

\medskip

\begin{lemma}
The fiber bundle $\VinBun_{G,P} \to T_{adj,P}^+$ is trivial.
\end{lemma}

\begin{proof}
The subgroup
$$\prod_{i \in \CI_M} \BG_m \ \ \longinto \ \ \prod_{i \in \CI} \BG_m \ = \ T_{adj}$$
acts simply transitively on $T_{adj,P}^+$; lifting this action to $\VinBun_{G,P}$ thus trivializes this fiber bundle.
\end{proof}

\bigskip

\ssec{Compactification of the diagonal}
By Subsection \ref{Stratification by parabolics} above the space $\barBun_G$ contains the open substack
$$\Bun_G \, \times \, \cdot/Z_G \ = \ \barBun_{G,G} \ \ \longinto \ \ \barBun_G \, ;$$
In particular we obtain a natural map
$$b: \ \Bun_G \ \longto \ \barBun_G$$
which forms a $Z_G$-bundle over its image $\barBun_{G,G}$ in $\barBun_G$. Furthermore, by construction the space $\barBun_G$ admits a natural forgetful map
$$\bar\Delta: \ \barBun_G \ \longto \ \Bun_G \times \Bun_G \, ,$$
yielding a factorization of the diagonal morphism $\Delta$ of $\Bun_G$ as
$$\xymatrix@+10pt{
\Bun_G \ar^{b}[r] \ar@/_-2pc/[rr]^{\Delta} & \barBun_G \ar^{\bar\Delta \ \ \ \ \ }[r] & \Bun_G \times \Bun_G \, .
}$$

\medskip

The space $\barBun_G$ is a compactification of $\Bun_G$ in the sense of the following remark, which we will neither use nor prove in the present article:

\begin{remark}
The map $\bar\Delta$ is proper.
\end{remark}

\medskip

\ssec{The defect-free locus}

We define the \textit{defect-free locus} of $\VinBun_G$ to be the open substack
$$\sideset{_0}{_G}\VinBun \ \ := \ \ \Maps(X, \sideset{_0}{_G}\Vin / G \times G) \, .$$
Lemma \ref{Vinberg strata} above implies that
$$\sideset{_0}{_{G}}\VinBun |_{c_P} \ \ = \ \ \Bun_{P^-} \underset{\Bun_M}{\times} \Bun_P \, .$$

Furthermore we have:

\medskip

\begin{proposition}
The restriction of the map $v$ to the defect-free locus
$$v: \ \sideset{_0}{_G}\VinBun \ \ \longto \ \ T_{adj}^+$$
is smooth. In particular, the defect-free locus $\sideset{_0}{_G}\VinBun$ itself is smooth.
\end{proposition}

\begin{proof}
The proof given in the case $G = \SL_2$ in \cite[Proposition 2.2.3]{Sch1} carries over without change. Indeed, the proof in \cite{Sch1} is given in the language of mapping stacks, and the only group-theoretic input in the proof is the fact that the stabilizers of the $G \times G$-action on the fibers of the map $\sideset{_0}{_G}\Vin \to T_{adj}^+$ are smooth; for an arbitrary reductive group $G$, this is established in \cite[D.4.6]{DrG2}. The proof from \cite{Sch1} then applies verbatim.
\end{proof}

\bigskip
\bigskip

\section{The defect stratification}
\label{The defect stratification}

\ssec{Recollections}

In this section we construct natural stratifications of the loci $\VinBun_{G,P}$. To do so, we first recall:

\sssec{The monoid $\overline{M}$}
\label{The monoid barM}

Let $P$ be a parabolic of $G$ and let $M$ be its Levi quotient. We now recall the definition of a certain reductive monoid $\overline{M}$ containing $M$ as a dense open subgroup; the definition of $\overline{M}$ depends on the realization of $M$ as a Levi of $G$. We refer the reader to \cite{BG1} and \cite{W1} for proofs and additional background.

\medskip

As before let $U_P$ denote the unipotent radical of $P$. Recall from e.g. \cite{BG1} that the quotient $G/U_P$ is strongly quasi-affine; we denote by $\overline{G/U_P}$ its affine closure. We then define $\overline{M}$ as the closure of $M$ inside $\overline{G/U_P}$ under the embedding
$$M = P/U_P \ \longinto \ G/U_P \ \subset \ \overline{G/U_P} \, .$$
The $M$-actions from the left and from the right on $G/U(P)$ induce $M$-actions from the left and from the right on $\overline{M}$, which in turn extend to $\overline{M}$-actions; thus $\overline{M}$ forms an algebraic monoid containing the group $M$. Alternatively, one can also define $\overline{M}$ as follows: Consider not the tautological embedding of $M = P^-/U_{P^-}$ into $G/U_{P^-}$ but rather the embedding given by the inverse:
$$M \ \longinto \ G/U_{P^-} \, , \ \ \ m \ \longmapsto \ m^{-1}$$
Using this embedding, one can then also define $\overline{M}$ as the closure of $M$ inside $\overline{G/U_{P^-}}$.

\medskip

\sssec{Embedding of $\overline{M}$ into $\Vin_G$}
\label{Embedding of barM into Vin_G}
Next recall that the embeddings of the first factor
$$G/U_P \ \longinto \ (G/U_P \times G/U_{P^-})/M$$
and the second factor
$$G/U_{P^-} \ \longinto \ (G/U_P \times G/U_{P^-})/M$$
extend to closed immersions
$$\overline{G/U_P} \ \ \longinto \ \ \overline{(G/U_P \times G/U_{P^-})/M}$$
and
$$\overline{G/U_{P^-}} \ \ \longinto \ \ \overline{(G/U_P \times G/U_{P^-})/M} \, .$$
Consider the two closed embeddings
$$\overline{M} \ \ \longinto \ \ \overline{(G/U_P \times G/U_{P^-})/M} \ = \ \Vin_G|_{c_P}$$
of $\overline{M}$ obtained by composing the previous embeddings with the embeddings of $\overline{M}$ into $\overline{G/U_P}$ and $\overline{G/U_{P^-}}$ from \ref{The monoid barM} above. Then one can show that these two embeddings of $\overline{M}$ into $\Vin_G|_{c_P}$ agree. Furthermore, this embedding is $M \times M$-equivariant for the natural $M \times M$-action on $\overline{M}$ and the $M \times M$-action on $\Vin_G|_{c_P}$ obtained by restricting the $G \times G$-action to the Levi subgroup $M \times M$.

\medskip

One can show (see \cite{W1}):

\begin{lemma}
\label{bar M lemma}
\begin{itemize}
\item[]
\item[(a)] The variety $\overline{M}$ is normal.
\item[(b)] The composition
$$\ \ \ \ \ \overline{M} \longinto \overline{G/U_P} = \Spec \bigl( \Gamma(G, \CO_G)^{U_P} \bigr) \longto \Spec \bigl( \Gamma(G, \CO_G)^{U_P \times U_{P^-}} \bigr)$$
is an isomorphism respecting the natural $M \times M$-actions.
\item[(c)] The composition
$$\ \ \ \ \ \overline{M} \ \longinto \ \Vin_G|_{c_P} \ \longto \ \Spec \bigl( \Gamma(\Vin_G|_{c_P}, \CO_{\Vin_G|_{c_P}})^{U_P \times U_{P^-}} \bigr)$$
is an isomorphism respecting the natural $M \times M$-actions.
\end{itemize}
\end{lemma}

\medskip

\sssec{Spaces of effective divisors}

Let $\thetacheck \in \Lambdach_{G,P}^{pos}$. There exist unique non-negative integers $n_i \in \BZ_{\geq 0}$ such that $\check\theta$ is the image of $\sum_{i \in \CI \setminus \CI_M} n_i \alphacheck_i$ under the natural projection map $\Lambdach_G \onto \Lambdach_{G,P}$.
Then we define
$$X^{\thetacheck} \ \ = \ \ \prod_{i \in \CI} X^{(n_i)} \, .$$
Thus as a variety, the space $X^{\thetacheck}$ is a partially symmetrized power of the curve $X$. Its points can be thought of as $\Lambdach_{G,P}^{pos}$-valued divisors on $X$, i.e., as formal linear combinations $\sum_{k} \thetacheck_k x_k$ with $x_k \in X$ and $\thetacheck_k \in \Lambdach_{G,P}^{pos}$ satisfying $\sum_k \thetacheck_k = \thetacheck$.

\medskip

\sssec{The $G$-positive Hecke stack for $M$}
\label{The G-positive Hecke stack for M}

Recall that the \textit{$G$-positive Hecke stack of $M$} is defined as the mapping stack
$$\CH_{M, G-pos} \ \ := \ \ \Maps_{gen}(X, \, \overline{M}/ M \times M \, \supset \, \cdot / M) \, .$$
By construction the \textit{$G$}-positive Hecke stack admits a forgetful map
$$\CH_{M, G-pos} \ \ \longto \ \ \Bun_M \times \Bun_M \, .$$
As the connected components of $\Bun_M$ are indexed by $\pi_0(\Bun_M) = \Lambdach_{G,P}$, we obtain a disjoint union decomposition
$$\CH_{M, G-pos} \ \ = \ \ \bigcup_{\lambdach_1, \lambdach_2} \CH^{\lambdach_1, \lambdach_2}_{M, G-pos}$$
where the disjoint union runs over all $\lambdach_1, \lambdach_2 \in \Lambdach_{G,P}$ such that $\lambdach_1 \leq \lambdach_2$.

\medskip

\sssec{The $G$-positive affine Grassmannian for $M$}
\label{The G-positive affine Grassmannian for M}

Fixing a trivialization of one of the two $M$-bundles appearing in the definition of the $G$-positive Hecke stack $\CH_{M, G-pos}$ above we obtain the \textit{$G$-positive part of the Beilinson-Drinfeld affine Grassmannian of $M$}, which we denote by $\Gr_{M, G-pos}$. In other words, we define
$$\Gr_{M, G-pos} \ \ := \ \ \Maps_{gen}(X, \, \overline{M}/M \, \supset \, M/M = pt) \, .$$

\medskip

\sssec{Maps to spaces of effective divisors}
\label{Maps to spaces of effective divisors}
We denote by $T_M$ the torus
$$T_M \ := \ M/[M,M] \ = \ P/[P,P] \, .$$
Recall from \cite{BG1} that the quotient $G/[P,P]$ is strongly quasi-affine, and let $\overline{G/[P,P]}$ denote its affine closure. Let $\overline{T_M}$ denote the closure of $T_M$ in $\overline{G/[P,P]}$ under the natural embedding
$$T_M \ = \ P/[P,P] \ \longinto \ G/[P,P] \ \subset \ \overline{G/[P,P]} \, .$$
The action of $T_M$ on itself by left or right translation extends to an action on $\overline{T_M}$, and the mapping stack $\Maps_{gen}(X, \overline{T_M}/T_M \supset T_M/T_M = pt)$ admits a disjoint union decomposition into connected components
$$\Maps_{gen}(X, \overline{T_M}/T_M \, \supset \, T_M/T_M = pt) \ \ = \ \ \bigcup_{\check\theta \in \Lambdach_{G,P}^{pos}} X^{\check\theta} \, .$$

\medskip

The projection map $M \onto M/[M,M]$ extends to a map $\overline{M} \to \overline{T_M}$ which is compatible with the natural actions of $M \times M$ and $T_M \times T_M$. In particular we obtain a natural map
$$\Gr_{M, G-pos} = \Maps_{gen}(X, \overline{M}/M \supset pt) \longto \Maps_{gen}(X, \overline{T_M}/T_M \supset pt) = \bigcup_{\check\theta \in \Lambdach_{G,P}^{pos}} X^{\check\theta} \, .$$
We denote the inverse image of the connected component $X^{\check\theta}$ under this map by $\Gr^{\check\theta}_{M, G-pos}$.

\medskip

\sssec{Factorization of $\Gr_{M, G-pos}$}
\label{Factorization of Gr}

The collection of maps
$$\Gr^{\check\theta}_{M, G-pos} \ \longto \ X^{\check\theta}$$
from Subsection \ref{Maps to spaces of effective divisors} above satisfies the following \textit{factorization property}:
Let $\check\theta_1, \check\theta_2 \in \Lambdach_{G,P}^{pos}$ and let $\check\theta := \check\theta_1 + \check\theta_2$. Then the natural map
$$X^{\check\theta_1} \stackrel{\circ}{\times} X^{\check\theta_2} \ \longto \ X^{\check\theta}$$
defined by adding effective divisors induces the following cartesian square:
$$\xymatrix@+10pt{
\Gr^{\check\theta_1}_{M, G-pos} \stackrel{\circ}{\times} \Gr^{\check\theta_2}_{M, G-pos} \ar[r] \ar[d]   &   \Gr^{\check\theta}_{M, G-pos} \ar[d]      \\
X^{(\check\theta_1)} \stackrel{\circ}{\times} X^{(\check\theta_2)} \ar[r]         &       X^{(\check\theta)}       \\
}$$

\medskip

\ssec{The stratification}

\sssec{Strata maps}
\label{Strata maps}
The closed immersion
$\overline{M} \, \longinto \, \overline{(G/U_P \times G/U_{P^-})/M}$
from Subsection \ref{Embedding of barM into Vin_G} above induces a map of quotient stacks
$$\overline{M} / P \times P^- \ \ \longto \ \ \Bigl( \overline{(G/U_P \times G/U_{P^-})/M} \Bigr) / G \times G \, ,$$
which by Lemma \ref{Vinberg strata} in turn induces a map
$$f: \ \Maps_{gen}(X, \, \overline{M} / P \times P^- \, \supset \, M / P \times P^-) \ \ \longto \ \ \VinBun_G|_{c_P} \, .$$
Rewriting the quotient stack $ \overline{M} / P \times P^-$ as
$$ \overline{M} / P \times P^- \ \ = \ \  \cdot / P^{-} \ \underset{\cdot / M}{\times} \ \overline{M} / M \times M \ \underset{\cdot / M}{\times} \ \cdot / P \, ,$$
the disjoint union decomposition of the $G$-positive Hecke stack in Subsection \ref{The G-positive Hecke stack for M} above implies that the source of the map $f$ decomposes into a disjoint union
$$\bigcup_{(\lambdach_1, \lambdach_2)} \Bun_{P^-, \lambdach_1} \ \underset{\Bun_M}{\times} \ \CH^{\lambdach_1, \lambdach_2}_{M, G-pos} \ \underset{\Bun_M}{\times} \ \Bun_{P,\lambdach_2} \, ,$$
where $\lambdach_1, \lambdach_2 \in \Lambdach_{G,P} = \pi_0(\Bun_{P^-}) = \pi_0(\Bun_{P})$ and $\lambdach_1 \leq \lambdach_2$. We denote by $f_{\lambdach_1, \lambdach_2}$ the restriction of $f$ to the corresponding substack in the above decomposition.

\medskip

We will show that for any parabolic $P$ the fiber $\VinBun_G|_{c_P}$ admits the following \textit{defect stratification}:

\medskip

\begin{proposition}
\label{defect stratification proposition}
\begin{itemize}
\item[]
\item[]
\item[(a)] The map $f_{\lambdach_1, \lambdach_2}$ is a locally closed immersion. We denote the corresponding locally closed substack by
$$\sideset{_{\lambdach_1, \lambdach_2}}{_G}\VinBun|_{c_P} \ \longinto \ \VinBun_G|_{c_P} \, .$$
\item[(b)] The locally closed substacks $\sideset{_{\lambdach_1, \lambdach_2}}{_G}\VinBun|_{c_P}$ form a stratification of $\VinBun_G|_{c_P}$, i.e.: On the level of $k$-points the stack $\VinBun_G|_{c_P}$ is equal to the disjoint union
$$\VinBun_G|_{c_P} \ \ = \ \ \bigcup_{(\lambdach_1, \lambdach_2)} \ \sideset{_{\lambdach_1, \lambdach_2}}{_{G,P}}\VinBun \, ,$$
where the union runs over all $\lambdach_1, \lambdach_2 \in \Lambdach_{G,P}$ such that $\lambdach_1 \leq \lambdach_2$.
\end{itemize}
\end{proposition}

\medskip

We will prove Proposition \ref{defect stratification proposition} in Subsection \ref{Compactifying the strata maps} below by compactifying the strata maps $f_{\lambdach_1, \lambdach_2}$. Before doing so we introduce the following terminology:

\sssec{Defect value and defect}

Each stratum
$$\sideset{_{\lambdach_1, \lambdach_2}}{_{G,P}}\VinBun \ = \ \Bun_{P^-, \lambdach_1} \ \underset{\Bun_M}{\times} \ \CH^{\lambdach_1, \lambdach_2}_{M, G-pos} \ \underset{\Bun_M}{\times} \ \Bun_{P,\lambdach_2}$$
admits a forgetful map to the stack $\CH^{\lambdach_1, \lambdach_2}_{M, G-pos}$. Given a $k$-point of $\VinBun_G|_{c_P}$ lying in this stratum, we refer to the corresponding $k$-point of $\CH^{\lambdach_1, \lambdach_2}_{M, G-pos}$ as its \textit{defect value} and to the positive coweight $\check\theta := \lambdach_2 - \lambdach_1 \in \Lambdach_{G,P}^{pos}$ as its \textit{defect}.

\bigskip

\ssec{Compactifying the strata maps}
\label{Compactifying the strata maps}

\sssec{Recollections on Drinfeld's compactifications $\tildeBun_P$}
\label{Recollections on Drinfeld's compactifications tildeBun_P}

We now briefly recall Drinfeld's compactifications $\tildeBun_P$; we refer the reader to \cite{BG1} for proofs and background.
The space $\tildeBun_P$ is defined as the mapping stack
$$\tildeBun_P \ := \ \Maps_{gen}(X, \, G \backslash \overline{G/U(P)} / M \, \supset \, \cdot / P) \, .$$
It naturally contains $\Bun_P$ as a dense open substack. The natural schematic map $\Fp: \Bun_P \to \Bun_G$ extends to a schematic map
$$\bar \Fp: \ \tildeBun_P \ \longto \ \Bun_G$$
which is proper when restricted to any connected component $\tildeBun_{P, \check\lambda}$ of $\tildeBun_P$, where $\check\lambda \in \pi_0(\tildeBun_P) = \Lambdach_{G,P}$.

\medskip

Finally, we recall that the space $\tildeBun_P$ admits the following stratification. The action map
$$G/U_P \times M \ \longto \ G/U_P$$
extends to an action map
$$\overline{G/U_P} \times \overline{M} \ \longto \ \overline{G/U_P}$$
of the monoid $\overline{M}$, which in turn induces a map
$$\overline{G/U_P} / G \times M \ \ \underset{\cdot / M}{\times} \ \ \overline{M} / M \times M \ \ \longto \ \ \overline{G/U_P} / G \times M \, .$$
Passing to mapping stacks we obtain, for any $\check\lambda \in \Lambdach_{G,P}$ and $\check\theta \in \Lambdach_{G,P}^{pos}$, natural maps
$$\CH^{\lambdach_1, \lambdach_2}_{M, G-pos} \underset{\Bun_M}{\times} \tildeBun_{P, \check\lambda + \check\theta} \ \ \longto \ \ \tildeBun_{P, \check\lambda} \, .$$
One can then show (see \cite{BG1}, \cite{BFGM}) that the restricted maps
$$\CH^{\lambdach_1, \lambdach_2}_{M, G-pos} \underset{\Bun_M}{\times} \Bun_{P, \check\lambda + \check\theta} \ \ \longto \ \ \tildeBun_{P, \check\lambda}$$
are locally closed immersions, and that they stratify $\tildeBun_{P, \check\lambda}$ as $\check\theta$ ranges over the set $\Lambdach_{G,P}^{pos}$:
$$\tildeBun_{P, \check\lambda} \ \ = \ \ \bigcup_{\check\theta \in \Lambdach_{G,P}^{pos}} \ \CH^{\lambdach_1, \lambdach_2}_{M, G-pos} \underset{\Bun_M}{\times} \Bun_{P, \check\lambda + \check\theta}$$

\medskip

\sssec{Compactifying the maps $f_{\lambdach_1, \lambdach_2}$}
Recall that the fiber
$$\Vin_G|_{c_P} \ \ = \ \ \overline{(G/U_P \times G/U_{P^-})/M}$$
carries a structure of semigroup (without unit) and naturally contains, by Subsection \ref{Embedding of barM into Vin_G} above, the varieties $\overline{G/U_P}$, $\overline{M}$, and $\overline{G/U_{P^-}}$ as subvarieties. We can thus define a map
$$\overline{G/U_P} \times \overline{M} \times \overline{G/U_{P^-}} \ \ \longto \ \ \Vin_G|_{c_P} = \overline{(G/U_P \times G/U_{P^-})/M}$$
by multiplying these three subvarieties. Alternatively, one can first act by $\overline{M}$ on either $\overline{G/U_P}$ or $\overline{G/U_{P^-}}$ and then multiply; this yields the same map.

\medskip

The above map gives rise to a map
$$\overline{G/U_{P^-}} / G \times M \ \ \underset{\cdot / M}{\times} \ \ \overline{M} / M \times M \ \ \underset{\cdot / M}{\times} \ \ \overline{G/U(P^-)} / G \times M \ \ \ \longto \ \ \ \Vin_G|_{c_P} / G \times G \, ,$$
which in turn induces maps
$$\bar f_{\lambdach_1, \lambdach_2}: \ \tildeBun_{P^-, \lambdach_1} \ \underset{\Bun_M}{\times} \ \CH^{\lambdach_1, \lambdach_2}_{M, G-pos} \ \underset{\Bun_M}{\times} \ \tildeBun_{P,\lambdach_2} \ \ \longto \ \VinBun_G|_{c_P} \, .$$
The maps $\bar f_{\lambdach_1, \lambdach_2}$ extend the strata maps $f_{\lambdach_1, \lambdach_2}$ from Subsection \ref{Strata maps} above, and the properness of $\tildeBun_P$ and $\tildeBun_{P^-}$ over $\Bun_G$ implies that the maps $\bar f_{\lambdach_1, \lambdach_2}$ are proper as well.

\medskip

\sssec{Proof of stratification results}
\begin{proof}[Proof of Proposition \ref{defect stratification proposition}]

\mbox{}

\medskip

\noindent \textit{Step 1: Set-theoretic stratification.}
We first claim that the map
$$\overline{M} / P \times P^- \ \ \longto \ \ \Bigl( \overline{(G/U_P \times G/U_{P^-})/M} \Bigr) / G \times G \, ,$$
used to define the strata maps is proper. Indeed, this follows from the fact that
$\overline{M}$ is closed in $\overline{(G/U_P \times G/U_{P^-})/M}$ and the fact that $P$ and $P^-$ are parabolic subgroups of $G$. Furthermore, this map becomes an isomorphism when restricted to the interior loci:
$$M / P \times P^- \ \ \stackrel{\cong}{\longto} \ \ \bigl((G/U_P \times G/U_{P^-})/M \bigr) / G \times G$$
Applying the mapping stack construction with the requirement of generic factorization through the interior loci to the above map of quotient stacks yields the disjoint union $f = \coprod f_{\lambdach_1, \lambdach_2}$ of the strata maps. Now the valuative criterion of properness shows that the map $f$ is a bijection on the level of $k$-points: The injectivity follows from the uniqueness part of the criterion, and the surjectivity from the existence part of the criterion. This establishes the stratification on the set-theoretic level, and will complete the proof of the Proposition once we show that the maps $f_{\lambdach_1, \lambdach_2}$ are indeed locally closed immersions.

\medskip

\noindent \textit{Step 2: The monomorphism property.}
To show that the map $f_{\lambdach_1, \lambdach_2}$ is a locally closed immersion, we first show that it is a monomorphism. To do so, first note that the closed immersion
$$\overline{M} \ \ \longinto \ \ \overline{(G/U_P \times G/U_{P^-})/M}$$
induces a closed immersion
$$\overline{M} / P \times P^- \ \ \longinto \ \ \Bigl( \overline{(G/U_P \times G/U_{P^-})/M} \Bigr) / P \times P^- \, ,$$
through which the map
$$\overline{M} / P \times P^- \ \ \longto \ \ \Bigl( \overline{(G/U_P \times G/U_{P^-})/M} \Bigr) / G \times G$$
inducing the strata maps $f_{\lambdach_1, \lambdach_2}$ factors.
Thus it suffices to show that, given an $X \times S$-point of $\overline{M} / P \times P^-$, the corresponding $P$-bundle and $P^-$-bundle on $X \times S$ are uniquely determined by the induced $X \times S$-point of $\Bigl( \overline{(G/U_P \times G/U_{P^-})/M} \Bigr) / G \times G$. We show this for the corresponding $P^-$-bundle; the case of the $P$-bundle is analogous.

\medskip

To do so, let for any dominant weight $\lambda \in \Lambda_G^+$ denote by $V^\lambda$ the corresponding \textit{Weyl module} of $G$, i.e., the module
$$V^\lambda \ := \ H^0 (G/B, \CO(- w_0 \lambda))^* \, ;$$
here $G/B$ denotes the flag variety of $G$, and $\CO(- w_0 \lambda))$ denotes the line bundle on $G/B$ corresponding to the dominant weight $-w_0(\lambda) \in \Lambda_G^+$.
Next recall from e.g. \cite[Ch. 1]{BG1}, \cite[Prop. 3.2.8]{Sch2} that, on any scheme, the datum of a reduction $F_{P^-}$ of a $G$-bundle $F_G$ to $P^-$ gives rise to, for each $\lambda \in \Lambda_G^+$, a surjection of associated vector bundles
$$V^\lambda_{F_G} \ \longonto \ (V^\lambda_{U_{P^-}})_{F_{P^-}} \, ,$$
and that conversely any $P^-$-bundle is uniquely determined by this collection of quotient vector bundles. We will now show that the collection of quotient vector bundles corresponding to the $P^-$-bundle on $X \times S$ under consideration above is indeed uniquely determined by the induced $X \times S$-point of $\bigl((G/U_P \times G/U_{P^-})/M \bigr) / G \times G$.

\medskip

By the definition of $\Vin_G$ via Tannakian formalism in Subsection \ref{The definition of Vin_G via classification of reductive monoids} above, the monoid $\Vin_G$ admits a $G \times G$-equivariant monoid homomorphism $\Vin_G \to \End(V^\lambda)$, for any $\lambda \in \Lambda_G^+$. In particular, any $X \times S$-point of $\Vin_G|_{c_P} / G \times G$ with corresponding $G$-bundles $F_G^1, F_G^2$ gives rise, for each $\lambda \in \Lambda_G^+$, to a map of vector bundles
$$V^\lambda_{F_G^1} \ \stackrel{\beta_\lambda}{\longto} \ V^\lambda_{F_G^2}$$
on $X \times S$.
But by definition of the map $f_{\lambdach_1, \lambdach_2}$, the surjection $V^\lambda_{F_G^1} \onto (V^\lambda_{U_{P^-}})_{F_{P^-}}$ agrees with the surjection
$$V^\lambda_{F_G^1} \ \stackrel{\beta_\lambda}{\longonto} \ \im(\beta_{\lambda}) \ \longinto V^\lambda_{F_G^2} \, .$$
This shows that the induced $X \times S$-point of $\Bigl( \overline{(G/U_P \times G/U_{P^-})/M} \Bigr) / G \times G$ uniquely determines the $P^-$-bundle, as desired.

\medskip

\noindent \textit{Step 3: Locally closed immersion.}
We can now show that $f_{\lambdach_1, \lambdach_2}$ is indeed a locally closed immersion. To do so, we denote by $\CB$ the boundary of

$$\tildeBun_{P^-, \lambdach_1} \ \underset{\Bun_M}{\times} \ \CH^{\lambdach_1, \lambdach_2}_{M, G-pos} \ \underset{\Bun_M}{\times} \ \tildeBun_{P,\lambdach_2} \, ,$$

\noindent i.e., the closed complement of the open substack

$$\Bun_{P^-, \lambdach_1} \ \underset{\Bun_M}{\times} \ \CH^{\lambdach_1, \lambdach_2}_{M, G-pos} \ \underset{\Bun_M}{\times} \ \Bun_{P,\lambdach_2} \, .$$

\medskip

\noindent The image of the boundary $\CB$ under the map $\bar f_{\lambdach_1, \lambdach_2}$ is a closed substack of $\VinBun_G$ since the map $\bar f_{\lambdach_1, \lambdach_2}$ is proper; let $\CU$ denote its open complement. We claim that taking the inverse image of $\CU$ under $\bar f_{\lambdach_1, \lambdach_2}$ yields the following cartesian square:
$$\xymatrix@+10pt{
\Bun_{P^-, \lambdach_1} \underset{\Bun_M}{\times} \CH^{\lambdach_1, \lambdach_2}_{M, G-pos} \underset{\Bun_M}{\times} \Bun_{P,\lambdach_2} \ar@{^(->}[r]^{\text{open}} \ar[d] \ar[rd]^{f_{\lambdach_1, \lambdach_2}} \ & \ \tildeBun_{P^-, \lambdach_1} \underset{\Bun_M}{\times} \CH^{\lambdach_1, \lambdach_2}_{M, G-pos} \underset{\Bun_M}{\times} \tildeBun_{P,\lambdach_2} \ar[d]^{\bar f_{\lambdach_1, \lambdach_2}} \\
\CU \ \ar@{^(->}[r]^{\text{open}}& \ \VinBun_G|_{c_P} \\
}$$

\medskip
\medskip

\noindent This follows from the fact that any point of $\VinBun_G|_{c_P}$ lying in the image of the boundary $\CB$ must, due to the stratification of $\tildeBun_P$ reviewed in Subsection \ref{Recollections on Drinfeld's compactifications tildeBun_P} above, have defect strictly greater than $\check\theta = \lambdach_2 - \lambdach_1$.

\medskip

The diagonal map of the above square is equal to the map $f_{\lambdach_1, \lambdach_2}$, which we have already shown to be a monomorphism. Hence the left vertical arrow is also a monomorphism. Being the base change of the proper map $\bar f_{\lambdach_1, \lambdach_2}$, the left vertical arrow is also proper, and thus it must be a closed immersion. This produces the desired factorization of the map $f_{\lambdach_1, \lambdach_2}$, showing that it is indeed a locally closed immersion.
\end{proof}

\bigskip
\bigskip
\bigskip

\section{Statements of theorems -- Geometry}
\label{Statements of theorems -- Geometry}

\bigskip

\ssec{Recollections}
\label{Recollections}

\sssec{Notation}
\label{Recollections notation}

For a scheme or stack $Y$ together with a map $Y \to \BA^1$ we denote by
$$\Psi: \ \D(Y|_{\BA^1 \setminus \{0\}}) \ \longto \ \D(Y|_{\{0\}})$$
the unipotent nearby cycles functor in the perverse and Verdier-self dual renormalization, i.e., we shift and twist the usual unipotent nearby cycles functor by $[-1](-\tfrac{1}{2})$. In this renormalization the functor $\Psi$ is t-exact for the perverse t-structure and commutes with Verdier duality literally and not just up to twist. We refer to $\Psi$ simply as \textit{the nearby cycles}. We refer the reader to \cite{gluing perverse sheaves} and \cite[Sec. 5]{BB} for background on unipotent nearby cycles.

\medskip

\sssec{The complex $\widetilde \Omega_P$}
\label{The complex tilde-Omega_P}
Let $P$ be a parabolic of $G$ and let $\lambdach_1, \lambdach_2 \in \Lambdach_{G,P}$ with $\lambdach_1 \leq \lambdach_2$. We now recall the definition of a certain complex $\widetilde \Omega_P^{\lambdach_1, \lambdach_2}$ on the $G$-positive part of the Hecke stack $\CH^{\lambdach_1, \lambdach_2}_{M, G-pos}$. To do so, let ${}_0Z^{P, \lambdach_1, \lambdach_2}_{rel}$ denote the \textit{open relative Zastava space} from \cite{BFGM} with degrees $\lambdach_1, \lambdach_2$; we recall its definition in Subsection \ref{Nearby cycles theorem} below.

\medskip

The stack ${}_0Z^{P, \lambdach_1, \lambdach_2}_{rel}$ is smooth and comes equipped with a natural map
$$\pi_Z: \ \ {}_0Z^{P, \lambdach_1, \lambdach_2}_{rel} \ \longto \ \CH^{\lambdach_1, \lambdach_2}_{M, G-pos} \, .$$
Let
$$\IC_{{}_0Z^{P, \lambdach_1, \lambdach_2}_{rel}} \ \ = \ \ \Qellbar_{{}_0Z^{P, \lambdach_1, \lambdach_2}_{rel}}[\dim {}_0Z^{P, \lambdach_1, \lambdach_2}_{rel}](\tfrac{1}{2} \dim {}_0Z^{P, \lambdach_1, \lambdach_2}_{rel})$$
denote the IC-sheaf of ${}_0Z^{P, \lambdach_1, \lambdach_2}_{rel}$. Then the complex $\widetilde \Omega_P^{\lambdach_1, \lambdach_2}$ is defined as the pushforward
$$\widetilde \Omega_P^{\lambdach_1, \lambdach_2} \ \ := \ \ \pi_{Z,!} \, \bigl( \IC_{{}_0Z^{P, \lambdach_1, \lambdach_2}_{rel}} \bigr) \, .$$
The statements of our main theorems will in fact rather involve the Verdier dual of $\widetilde \Omega_P^{\lambdach_1, \lambdach_2}$
$$\BD \, \widetilde \Omega_P^{\lambdach_1, \lambdach_2} \ \ = \ \ \pi_{Z,*} \, \bigl( \IC_{{}_0Z^{P, \lambdach_1, \lambdach_2}_{rel}} \bigr) \, .$$

\bigskip

\ssec{Main theorem about nearby cycles}

\sssec{Nearby cycles for various parabolics $P$}
We can now state our main theorem describing the stalks of the nearby cycles functors arising from the multi-parameter degeneration $\VinBun_G \to T_{adj}^+ = \BA^r$. Fix a parabolic $P$ of $G$ and consider the line
$$L_P = \BA^1 \ \longinto \ T_{adj}^+ = \BA^r$$
passing through the points $c_G$ and $c_P$ of $T_{adj}^+ = \BA^r$; here we identify the point $1 \in \BA^1$ with the point $c_G$ and the point $0 \in \BA^1$ with the point $c_P$. Let $\VinBun_G|_{L_P}$ denote the restriction of the family $\VinBun_G \to T_{adj}^+$ to the line $L_P = \BA^1$, and consider the nearby cycles functor associated to this one-parameter family. Let $\Psi_P \in D(\VinBun_G|_{c_P})$ denote the nearby cycles of the IC-sheaf
$$\IC_{\VinBun_G}|_{L_P} \ \ = \ \ \Qellbar[\dim \Bun_G + 1](\tfrac{1}{2} \dim \Bun_G + \tfrac{1}{2}) \, .$$
Then we have:

\medskip

\begin{theorem}
\label{nearby cycles theorem}
The $*$-restriction of $\Psi_P$ to the stratum
$$\sideset{_{\lambdach_1, \lambdach_2}}{_G}\VinBun|_{c_P} \ \ = \ \ \Bun_{P^-, \lambdach_1} \ \underset{\Bun_M}{\times} \ \CH^{\lambdach_1, \lambdach_2}_{M, G-pos} \ \underset{\Bun_M}{\times} \ \Bun_{P,\lambdach_2}$$
of $\VinBun_G|_{c_P}$ is equal to
$$\IC_{\Bun_{P^-, \lambdach_1}} \underset{\Bun_M}{\boxtimes} \ \BD \, \widetilde \Omega_P^{\lambdach_1, \lambdach_2} \underset{\Bun_M}{\boxtimes} \IC_{\Bun_{P, \lambdach_1}} \, .$$
\end{theorem}

\bigskip
\bigskip
\bigskip

\ssec{Main theorem about the $*$-extension of the constant sheaf}

\medskip

We now state our main theorem describing the $*$-stalks of the $*$-extension of the constant sheaf of the $G$-locus $\VinBun_{G,G}$. As will be clear from its formulation, this theorem is very closely related to the nearby cycles theorem, Theorem \ref{nearby cycles theorem} above; in fact, Theorem \ref{nearby cycles theorem} follows from a variant of Theorem \ref{ij} below. Thus Theorems \ref{nearby cycles theorem} and \ref{ij} will be proven simultaneously in Section \ref{Proofs II --- Sheaves} below.
To state the theorem, let $j_G$ denote open inclusion
$$j_G: \ \VinBun_{G,G} \ \longinto \ \VinBun_G$$
and let
$$\IC_{\VinBun_{G,G}} \ \ = \ \ \Qellbar_{\VinBun_{G,G}}[\dim \VinBun_{G,G}](\tfrac{1}{2} \dim \VinBun_{G,G})$$
denote the IC-sheaf of the $G$-locus. Then we have:

\medskip

\begin{theorem}
\label{ij}
The $*$-restriction of the $*$-extension $j_{G,*} \, \IC_{\VinBun_{G,G}}$ to the stratum
$$\sideset{_{\lambdach_1, \lambdach_2}}{_G}\VinBun|_{c_P} \ \ = \ \ \Bun_{P^-, \lambdach_1} \ \underset{\Bun_M}{\times} \ \CH^{\lambdach_1, \lambdach_2}_{M, G-pos} \ \underset{\Bun_M}{\times} \ \Bun_{P,\lambdach_2}$$
of the fiber $\VinBun_G|_{c_P}$ is equal to
$$\IC_{\Bun_{P^-, \lambdach_1}} \underset{\Bun_M}{\boxtimes} \ \Bigl( \BD \, \widetilde \Omega_P^{\lambdach_1, \lambdach_2} \otimes \bigl( H^*(\BA^1 \setminus \{0\})[1](\tfrac{1}{2}) \bigr)^{\otimes \, r - r_M} \Bigr) \underset{\Bun_M}{\boxtimes} \IC_{\Bun_{P, \lambdach_1}} \, .$$
\end{theorem}

\bigskip

\section{Statements of theorems -- Bernstein asymptotics}
\label{Statements of theorems -- Bernstein asymptotics}

The Bernstein asymptotics map is commonly treated in the principal case, i.e., for the Borel $B$ of $G$; for simplicity, we restrict to this case here as well. One can proceed analogously for the case of an arbitrary parabolic $P$ of $G$; a generalization of the proof below applies, and will be carried out in future work of Wang \cite{W2}.

\medskip

\ssec{The principal degeneration}

We specialize the discussion to the case $P=B$, i.e., to the family $\VinBun_G|_{L_B} \to L_B = \BA^1$. We will refer to this family as the \textit{principal degeneration} of $\Bun_G$ and will also denote it by $\VinBun_G^{princ}$; we denote the nearby cycles sheaf $\Psi_B$ also by $\Psi^{princ}$. We first repeat the basic definitions and statements in this notationally simpler case for the convenience of the reader. First, note that for $P=B$ we have $\Lambdach_{G,P} = \Lambdach_G$ and
$$\CH^{\lambdach_1, \lambdach_2}_{M, G-pos} \ = \ X^{\lambdach_2 - \lambdach_1} \times \Bun_{T, \lambdach_2} \, .$$
The stratification of the special fiber $\VinBun_G|_{c_B} = \VinBun_{G,B}$ thus takes the form
$$\VinBun_{G,B} \ \ = \ \ \bigcup_{(\lambdach_1, \check\theta, \lambdach_2)} \ \Bun_{B^-, \lambdach_1} \ \underset{\Bun_{T, \lambdach_1}}{\times} \ \bigl( \ X^{\check\theta} \times \Bun_{B, \lambdach_2} \bigr) \, ,$$
where $\lambdach_1, \lambdach_2 \in \Lambdach_G$, $\check\theta \in \Lambdach_G^{pos}$, and $\lambdach_2 - \check\theta = \lambdach_1$. Similarly to before let ${}_0Z^{B, \check\theta}$ denote the \textit{defect-free absolute Zastava space} for the Borel $B$ and a positive coweight $\check\theta \in \Lambdach_G^{pos}$ from \cite{FFKM}, \cite{BFGM}; see Subsection \ref{Recollections on Zastava spaces} below for the definition and for some basic properties. Let $\pi_Z: {}_0Z^{B, \check\theta} \to X^{\check\theta}$ denote the natural forgetful map, and as before define
$$\widetilde \Omega_B^{\check\theta} \ \ := \ \ \pi_{Z,!} \, \bigl( \IC_{{}_0Z^{B, \check\theta}} \bigr) \, .$$
Theorem \ref{nearby cycles theorem} above in this case then reads:

\begin{theorem}[Theorem \ref{nearby cycles theorem} in the principal case]
The $*$-restriction of $\Psi^{princ}$ to the stratum
$$\Bun_{B^-, \lambdach_1} \ \underset{\Bun_{T, \lambdach_1}}{\times} \ \bigl( X^{\check\theta} \times \Bun_{B, \lambdach_2} \bigr)$$
of $\VinBun_{G,B}$ is equal to
$$\IC_{\Bun_{B^-, \lambdach_1}} \underset{\Bun_T}{\boxtimes} \ \bigl( \BD \, \widetilde \Omega_B^{\check\theta} \ \boxtimes \ \IC_{\Bun_{B, \lambdach_1}} \bigr) \, .$$
\end{theorem}

\medskip

\ssec{Factorization of nearby cycles}

We first record that the nearby cycles sheaf $\Psi^{princ}$ \textit{factorizes} in the sense of Proposition \ref{factorization of nearby cycles proposition} below. This fact also follows a posteriori from the stalk computation of Theorem \ref{nearby cycles theorem} above together with the factorization of the Zastava spaces. However, we will obtain this fact as a byproduct of our study of the geometry of the family $\VinBun_G$; in particular, this fact will be established without reliance on the formula in Theorem \ref{nearby cycles theorem} above. An analogous result holds for the nearby cycles $\Psi_P$ associated with other parabolics $P$ of $G$ and is proven in the same fashion.

\medskip

\begin{proposition}[Factorization of nearby cycles]
\label{factorization of nearby cycles proposition}
Let $y$ be a point of the special fiber $\VinBun_{G,B}$ and let $D \in X^{\check\theta}$ be its defect value. Then the $*$-stalk of $\Psi^{princ}$ at the point $y$ depends only on its defect value $D$; we denote the corresponding stalk by $\Psi^{princ}|^*_D$. Furthermore, if $D_1, D_2 \in X^{\check\theta}$ are $\Lambdach_G^{pos}$-valued divisors whose supports form disjoint subsets of the curve $X$ and such that $D_1 + D_2 = D$, then we have
$$\Psi^{princ}|^*_D \; [\dim(\Bun_G)](\tfrac{1}{2} \dim(\Bun_G)) \ = \ \Psi^{princ}|^*_{D_1} \otimes \Psi^{princ}|^*_{D_2} \, .$$
\end{proposition}

\medskip

\ssec{The function corresponding to $\Psi^{princ}$}

Let $\BF_q$ be a finite field with $q$ elements, let $X$ be a smooth projective curve over $\BF_q$, let $G$ be a reductive group  over $\BF_q$, and consider $\Bun_G$ over $\BF_q$. We will now state a combinatorial formula describing the function corresponding to the sheaf $\Psi^{princ}$ under the sheaf-function dictionary. To do so, we first introduce the following notation:

\sssec{Kostant partitions}
For $\thetacheck \in \Lambdach_G^{pos}$ we define a \textit{Kostant partition} of $\thetacheck$ to be a collection of non-negative integers $(n_{\betacheck})_{\betacheck \in \check R^+}$ indexed by the set of positive coroots $\check R^+$ of $G$, satisfying that
$$\thetacheck \ = \ \sum_{\betacheck \in \check R^+} n_{\betacheck} \betacheck \, .$$
In other words, a Kostant partition of $\thetacheck$ is a partition $\thetacheck = \sum_k \thetacheck_k$ of $\thetacheck$ where each summand $\thetacheck_k$ is in fact a positive coroot of $G$. Abusing notation we will simply refer to the expression $\thetacheck = \sum_{\betacheck \in \check R^+} n_{\betacheck} \betacheck$ as a Kostant partition of $\thetacheck$.

\medskip

The finite set of all Kostant partitions of $\thetacheck$ will be denoted by $\Kost(\thetacheck)$. The cardinality of the set $\Kost(\thetacheck)$ is by definition the value of the Kostant partition function of the Langlands dual group $\check G$ evaluated at the weight $\thetacheck \in \Lambdach_G = \Lambda_{\check G}$ of $\check G$.

\medskip

\sssec{The statement}
For a Kostant partition $\CK$ of a positive coweight $\check\theta \in \Lambdach_G^{pos}$ we let $R_{\CK}$ denote the set of coroots $\check\beta$ appearing in $\CK$ with a non-zero coefficient $n_{\check\beta}$. Let $|R_{\CK}|$ denote the cardinality of $R_{\CK}$. Then we have:

\medskip

\begin{theorem}
\label{function corresponding to Psi}
\begin{itemize}
\item[]
\item[]
\item[(a)]
Let $y$ be an $\BF_q$-point of $\VinBun_{G,B}$ of defect value $\check\theta x$ for $x \in X(\BF_q)$. Then the trace of the geometric Frobenius on the $*$-stalk at $y$ of the nearby cycles $\Psi_{\VinBun_G^{princ}}$ is equal to
$$q^{-\tfrac{1}{2} \dim(\Bun_G)} \, q^{\langle \rho, \check\theta \rangle} \sum_{\CK \, \in \, \Kost(\check \theta)} (1-q)^{|R_{\CK}|} \; q^{-|K|} \, .$$
\item[]
\item[(b)]
Due to the factorization of $\Psi_{\VinBun_G^{princ}}$ in Proposition \ref{factorization of nearby cycles proposition} above, if $y$ has defect value $\sum_i \check\theta_i x_i$ for $x_i \in X(\BF_q)$ distinct, then the corresponding trace is equal to
$$q^{-\tfrac{1}{2} \dim(\Bun_G)} \, \prod_i \ q^{\langle \rho, \check\theta_i \rangle} \sum_{\CK \, \in \, \Kost(\check\theta_i)} (1-q)^{|R_{\CK}|} \; q^{-|K|} \, .$$
\end{itemize}
\end{theorem}

\bigskip

\ssec{Bernstein asymptotics and the conjecture of Sakellaridis}

We now recall the formulation of Sakellaridis's conjecture from \cite{Sak2}; we will deduce it from Theorem \ref{function corresponding to Psi} in Section \ref{Proofs III -- Bernstein asymptotics} below.
In the context of the Bernstein map discussed in the introduction we take the local field $F$ to be $\BF_q(\!(t)\!)$ with ring of integers $\CO = \BF_q(\!(t)\!)$, and consider the group $G(F)$ with its standard maximal compact subgroup $K=G(\CO)$.
We will use the notation from \cite{Sak1}, \cite{Sak2}, and \cite{SakV} and refer the reader to these sources for details. In particular we let $\phi_0 \in C^{\infty}(G(F))$ denote the ``'basic function'', i.e., the characteristic function of the standard maximal compact subgroup $K$, and let $\Asymp(\phi_0)$ denote its image under the asymptotics map. As in the above sources we will denote by $\mathbbold{1}_{\check\theta}$ the characteristic function of the $K \times K$-coset in $((G/N \times G/N^-)/T)(F)$ corresponding to a coweight $\check\theta \in \Lambdach_G$.

\medskip

Next let $x \in X(\BF_q)$ and let
$$\Tr(\Frob, \Psi_{\VinBun_G^{princ}}|^*_{\check\theta x})$$
denote the trace of the geometric Frobenius on the $*$-stalk of the nearby cycles $\Psi_{\VinBun_G^{princ}}$ at a point of defect value $\check\theta x$, multiplied by the normalization factor $q^{\dim(\Bun_G)/2}$, which is a result of our normalization of IC-sheaves. Then we have:

\medskip

\begin{theorem}[Sakellaridis's conjecture from \cite{Sak2}]
\label{theorem Sakellaridis's conjecture}
$$\Asymp(\phi_0) \ = \ \sum_{\check\theta \in \Lambdach_{G}^{pos}} \Tr(\Frob, \Psi_{\VinBun_G^{princ}}|^*_{\check\theta x}) \, \mathbbold{1}_{\check\theta}$$
\end{theorem}

\bigskip
\bigskip
\bigskip

\section{Proofs I --- Construction of local models for $\VinBun_G$}
\label{Proofs I --- Construction of local models for VinBun_G}

\bigskip

\ssec{Definition of the local models}

\sssec{Strict $P$-loci}
\label{Strict P-loci}
Throughout this section we fix a parabolic $P$ of $G$. Let $T_{adj, \geq P}^+$ denote the open subvariety of $T_{adj}^+$ formed by the union of those strata $T_{adj,Q}^+$ such that $P \subset Q$. Thus under the identification $T_{adj}^+ = \BA^r$ the open subvariety $T_{adj, \geq P}^+$ consists of those points $(c_i)_{i \in \CI}$ satisfying that $c_i \neq 0$ if $i \in \CI_M$.

\medskip

Let $T_{adj, \geq P, strict}^+$ denote the \textit{strict} version of $T_{adj, \geq P}^+$, defined as the closed subvariety
$$T_{adj, \geq P, strict}^+ \ \longinto \ T_{adj}^+ = \BA^r$$
obtained by requiring that $c_i = 1$ if $i \in \CI_M$. In particular $T_{adj, \geq P, strict}^+$ is itself an affine space $\BA^{r-r_M}$ of dimension $r-r_M$, and
$$T_{adj, \geq P, strict}^+ \ \subset \ T_{adj, \geq P}^+ \, .$$

\medskip

The affine space $T_{adj, \geq P, strict}^+$ admits a stratification
$$T_{adj, \geq P, strict}^+ \ \ = \ \ \bigcup_{Q \supseteq P} \ T_{adj, \geq P, strict, Q}^+ \, ,$$
indexed by parabolics $Q$ containing $P$, where $T_{adj, \geq, strict, Q}^+$ is defined as the intersection
$$T_{adj, \geq P, strict, Q}^+ \ \ : = \ \ T_{adj, \geq P, strict}^+ \ \cap \ T_{adj, Q}^+ \, .$$
By definition we have that $T_{adj, \geq P, strict, P}^+ \ = \ \{ c_P \}$ and that
$$T_{adj, \geq P, strict, G}^+ \ \cong \ (\BA^1 \setminus \{ 0 \})^{r-r_M} \ \subset \ \BA^{r-r_M} = T_{adj, \geq P, strict}^+$$

\medskip

Finally, we denote by $(\Vin_G)_{\geq P}$ and by $(\Vin_G)_{\geq P, strict}$ the inverse images of the corresponding loci in $T_{adj}^+$ under the map $\Vin_G \to T_{adj}^+$.

\medskip

\sssec{The open Bruhat locus}
We define \textit{the open Bruhat locus} $(\Vin_G)_{\geq P}^{Bruhat}$ in $(\Vin_G)_{\geq P}$ as the open subvariety obtained by acting by the subgroup $P \times U_{P^-} \subset G \times G$ on the section
$$\Fs: \ T_{adj, \geq P}^+ \ \longto \ (\Vin_G)_{\geq P} \, ,$$
i.e., we define $(\Vin_G)_{\geq P}^{Bruhat}$ as the open image of the map
$$P \times U_{P^-} \times T_{adj, \geq P}^+ \ \ \longto \ \ (\Vin_G)_{\geq P}$$
$$(p,u,t) \ \longmapsto \ (p,u) \cdot \Fs(t) \, .$$

\medskip

\noindent We define $(\Vin_G)_{\geq P, strict}^{Bruhat}$ analogously. Note that by definition the open Bruhat locus is contained in the non-degenerate locus, i.e.:
$$(\Vin_G)_{\geq P, strict}^{Bruhat} \ \ \subset \ \ (\Vin_G)_{\geq P, strict} \cap \sideset{_0}{_G}\Vin$$

\medskip

\sssec{GIT-quotient of the strict $P$-locus and the Bruhat locus}

The $G \times G$-action on $\Vin_G$ restricts to a $G \times G$-action on the strict locus $(\Vin_G)_{\geq P, strict}$, and thus induces a $G \times G$-action on its coordinate ring $k[(\Vin_G)_{\geq P, strict}]$. We now recall two lemmas about this action from \cite{W1}. The first one states that the GIT-quotient
$$(\Vin_G)_{\geq P, strict} \ // \ U_P \times U_{P^-} \ \ \ := \ \ \ \Spec(k[(\Vin_G)_{\geq P, strict}]^{U_P \times U_{P^-}})$$
of $(\Vin_G)_{\geq P, strict}$ by $U_P \times U_{P^-}$ is naturally isomorphic to $\overline{M} \times T_{adj, \geq P, strict}^+$, strengthening the assertion of part (c) of Lemma \ref{bar M lemma}:

\medskip

\begin{lemma}
\label{GIT quotient of strict locus}
The inclusion of the subring of $U_P \times U_{P^-}$-invariants
$$k[(\Vin_G)_{\geq P, strict}]^{U_P \times U_{P^-}} \ \longinto \ k[(\Vin_G)_{\geq P, strict}]$$
induces an $M$-equivariant map
$$(\Vin_G)_{\geq P, strict} \ \ \longto \ \ (\Vin_G)_{\geq P, strict} \ // \ U_P \times U_{P^-} \ = \ \overline{M} \times T_{adj, \geq P, strict}^+ \, .$$
The composition of this map with the projection onto the second factor recovers the usual map $(\Vin_G)_{\geq P, strict} \to T_{adj, \geq P, strict}^+$. The base change of this map along the inclusion $\overline{M} \times \{ c_P \} \into \overline{M} \times T_{adj, \geq P, strict}^+$ recovers the map from part (c) of Lemma \ref{bar M lemma}.
\end{lemma}

Over the open $M \subset \overline{M}$ we have:

\begin{lemma}
\label{GIT quotient of Bruhat locus}
The base change of the map from Lemma \ref{GIT quotient of strict locus} above along the inclusion $M \into \overline{M}$ yields a cartesian square
$$\xymatrix@+10pt{
(\Vin_G)_{\geq P, strict}^{Bruhat} \ar[r] \ar[d] & (\Vin_G)_{\geq P, strict} \ar[d]  \\
M \times T_{adj, \geq P, strict}^+ \ar[r] & \overline{M} \times T_{adj, \geq P, strict}^+ \\
}$$
in which all arrows are $M$-equivariant.
Furthermore, the left vertical arrow is a $U_P \times U_{P^-}$-torsor; thus we obtain an identification of the stack quotient
$$(\Vin_G)_{\geq P, strict}^{Bruhat}/U_P \times U_{P^-} \ \ \stackrel{\cong}{\longto} \ \ M \times T_{adj, \geq P, strict}^+ \, .$$
\end{lemma}

\medskip

\sssec{The definition of the local models}
We now define the local model for the $P$-locus as
$$Y^P \ \ := \ \ \Maps_{gen}\bigl(X, \, (\Vin_G)_{\geq P, strict} / P \times U_{P^-} \ \supset \ (\Vin_G)_{\geq P, strict}^{Bruhat} / P \times U_{P^-} \bigr) \, .$$
Note that by Lemma \ref{GIT quotient of strict locus} above the open substack used in this definition satisfies
$$(\Vin_G)_{\geq P, strict}^{Bruhat} / P \times U_{P^-} \ \ = \ \ T_{adj, \geq P, strict}^+ \, .$$

\medskip

\sssec{Natural maps}
Analogously to $\VinBun_G$, the local model $Y^P$ comes equipped with a map
$$v: \ Y^P \ \longto \ T_{adj, \geq P, strict}^+ \, .$$
Furthermore, by Lemmas \ref{GIT quotient of strict locus} and \ref{GIT quotient of Bruhat locus} above, the natural map from the stack quotient to the GIT quotient
$$(\Vin_G)_{\geq P, strict} \ / \ U_P \times U_{P^-} \ \ \ \longto \ \ \ (\Vin_G)_{\geq P, strict} \ // \ U_P \times U_{P^-}$$
induces a map
$$Y^P \ \longto \ \Gr_{M, G-pos} \times T_{adj, \geq P, strict}^+ \, .$$

\medskip

\noindent Let $Y^P \to \Gr_{M, G-pos}$ be the map obtained by composing the above map with the projection onto the factor $\Gr_{M, G-pos}$. Given an element $\check\theta \in \Lambdach_{G,P}^{pos}$ we define $Y^{P, \check\theta}$ as the inverse image of $\Gr_{M, G-pos}^{\check\theta}$ under this map; thus we obtain a map
$$\pi: \ Y^{P, \check\theta} \ \longto \ \Gr_{M, G-pos}^{\check\theta}$$
and, composing with the map $\Gr_{M, G-pos}^{\check\theta} \to X^{\check\theta}$, a map
$$Y^{P, \check\theta} \ \longto \ X^{\check\theta} \, .$$

\medskip

\sssec{Stratification by parabolics}
\label{Stratification by parabolics for local models}
The stratification of $T_{adj, \geq P, strict}^+$ indexed by parabolic subgroups $Q$ of $G$ containing the parabolic $P$ induces a stratification
$$Y^{P, \check\theta} \ \ = \ \ \bigcup_{Q \supseteq P} \ Y^{P, \check\theta}_Q \, .$$
Furthermore, exactly as in Subsection \ref{The T_adj-action on VinBun_G} above we have:

\begin{remark}
\label{triviality of fiber bundles for local models}
Let $Q$ be a parabolic containing $P$. Then the fiber bundle $Y^{P, \check\theta}_Q \to T_{adj,Q, strict}^+$ is trivial.
\end{remark}

Next let ${}_0Z^{P,\check\theta}$ denote the defect-free Zastava space from \cite{BFGM}; its definition is recalled in Subsection \ref{Recollections on Zastava spaces} below. Then directly from the definition of $Y^{P, \check\theta}$ we see:

\begin{remark}
\label{fiber of local model over 1}
The fiber $Y^{P, \check\theta}|_{c_G}$ of $Y^{P, \check\theta}$ over the point $c_G \in T_{adj, \geq P, strict}^+$ is naturally isomorphic to the the defect-free Zastava space ${}_0Z^{P,\check\theta}$.
\end{remark}

\bigskip

\ssec{Factorization in families}

Let $t \in T_{adj, \geq P, strict}^+$ and let $Y^P|_t$ denote the fiber of the map $v$ over the point $t$. Then
since
$$(\Vin_G)_{\geq P, strict}^{Bruhat} / P \times U_{P^-} \ \ = \ \ T_{adj, \geq P, strict}^+$$
by Lemma \ref{GIT quotient of strict locus} above, we find that
$$Y^P|_t \ \ = \ \ \Maps_{gen}\bigl(X, \, (\Vin_G|_t) / P \times U_{P^-} \supset pt \bigr)$$
parametrizes maps from the curve $X$ to the quotient $(\Vin_G|_t) / P \times U_{P^-}$ which generically on $X$ factor through the dense open point
$$\bigl((\Vin_G)_{\geq P, strict}^{Bruhat}\bigr)|_t / P \times U_{P^-} \ \ = \ \ pt \, .$$
In particular, this shows that the spaces $Y^{P, \check\theta}|_t$ are \textit{factorizable} with respect to the maps $Y^{P, \check\theta}|_t \to X^{\check\theta}$, in the sense of Subsection \ref{Factorization of Gr} above. In fact, the above shows the stronger statement that the local models $Y^{P, \check\theta}$ \textit{factorize in families} over $T_{adj, \geq P, strict}^+$ in the sense of the following lemma:

\medskip

\begin{lemma}
\label{factorization in families}
Let $\check\theta_1, \check\theta_2 \in \Lambdach_{G,P}^{pos}$ and let $\check\theta := \check\theta_1 + \check\theta_2$. Then the addition map of effective divisors
$$X^{\check\theta_1} \stackrel{\circ}{\times} X^{\check\theta_2} \ \longto \ X^{\check\theta}$$
induces the following cartesian square:
$$\xymatrix@+10pt{
Y^{P, \check\theta_1} \underset{ \ T_{adj, \geq P, strict}^+}{\stackrel{\circ}{\times}} Y^{P, \check\theta_2} \ \ar[rr] \ar[d] & & \ Y^{P, \check\theta} \ar[d] \\
X^{(\check\theta_1)} \stackrel{\circ}{\times} X^{(\check\theta_2)} \ \ar[rr] & & \ X^{(\check\theta)} \\
}$$

\end{lemma}

\bigskip

\ssec{Recollections on Zastava spaces}
\label{Recollections on Zastava spaces}

\sssec{The definition of parabolic Zastava space}
Let $P$ be a parabolic of $G$. Recall from \cite{BFGM} and \cite{FFKM} that the parabolic \textit{Zastava space} $Z^P$ is defined as
$$Z^P \ := \ \Maps_{gen}(X, \ (\overline{G/U_P}) / M \times U_{P^-} \ \supset \ pt)$$
where the dense open point corresponds to the open Bruhat cell $P \cdot U_{P^-} \subset G$. As is discussed in \cite{BFGM} and \cite{BG2}, the Zastava space forms a \textit{local model} for the space $\tildeBun_P$. Here we recall some relevant properties; we refer the reader to \cite{BFGM} and \cite{FFKM} for a more detailed treatment and proofs.

\medskip

\sssec{Basic properties}
\label{Zastava basic properties}
First, we recall that the open subspace
$${}_0Z^P \ \ := \ \ \Maps_{gen}(X, \ (G/U_P) / M \times U_{P^-} \ \supset \ pt)$$
of $Z^P$ is smooth.
Next, recall that the map $\overline{G/U_P} \to \overline{M}$ from Lemma \ref{bar M lemma} above induces a map $Z^P \to \Gr_{M, G-pos}$.
Similarly to above we denote by $Z^{P, \check\theta}$ the inverse image of $\Gr_{M, G-pos}^{\check\theta}$ under this map, and analogously for ${}_0Z^{P, \check\theta}$. By definition we obtain projection maps
$$Z^{P, \check\theta} \ \longto \ \Gr_{M, G-pos}^{\check\theta}$$
and
$$Z^{P, \check\theta} \ \longto \ X^{\check\theta} \, .$$
The spaces $Z^{P, \check\theta}$ are factorizable with respect to the maps $Z^{P, \check\theta} \to X^{\check\theta}$ in the sense of Subsection \ref{Factorization of Gr} above. 

\medskip

\sssec{Stratification}
\label{Zastava stratification}
The Zastava spaces $Z^{P, \check\theta}$ admit a \textit{defect stratification} analogous to the stratification of $\tildeBun_P$ discussed in Subsection \ref{Recollections on Drinfeld's compactifications tildeBun_P} above. Namely, as in Subsection \ref{Recollections on Drinfeld's compactifications tildeBun_P} above, the action map
$$\overline{M} \times G/U_P \ \longto \ \overline{G/U_P}$$
induces locally closed immersions
$$\CH^{-\check\theta, - \check\theta + \check\theta'}_{M, G-pos} \underset{\Bun_M}{\times} {}_0Z^{P, \check\theta - \check\theta'} \ \ \longinto \ \ Z^{P,\check\theta}$$
for any $\check\theta, \check\theta' \in \Lambdach_{G,P}^{pos}$ with $\check\theta' \leq \check\theta$. We denote the corresponding locally closed substack by ${}_{\check\theta'}Z^{P,\check\theta}$. Ranging over those $\check\theta' \in \Lambdach_{G,P}^{pos}$ satisfying $0 \leq \check\theta' \leq \check\theta$, the substacks ${}_{\check\theta'}Z^{P,\check\theta}$ form a stratification of $Z^{P, \check\theta}$:
$$Z^{P, \check\theta} \ \ = \ \ \bigcup_{0 \leq \check\theta' \leq \check\theta} \ {}_{\check\theta'}Z^{P,\check\theta}$$

\medskip

\sssec{Minor variants of Zastava space}
Below we will also consider the following two variants of the above Zastava space. First, we will consider the relative Zastava space
$$Z^P_{\Bun_M} \ := \ \Maps_{gen}(X, \ (\overline{G/U_P}) / M \times P^- \ \supset \ \cdot/M) \, .$$
It comes equipped with a forgetful map $Z^P_{\Bun_M} \to \Bun_M$ induced by the composite map
$$(\overline{G/U_P}) / M \times P^- \ \longto \ \cdot/P^- \ \longto \ \cdot/M \, ,$$
and the fiber of this forgetful map over the trivial $M$-bundle is precisely the Zastava space $Z^P$ considered above. The discussion from Subsections \ref{Zastava basic properties} and \ref{Zastava stratification} above carries over to this setting, and we use the analogous notation. Finally, given a coweight $\lambdach \in \Lambdach_{G,P} = \pi_0(\Bun_M)$ we denote by $Z^P_{{\Bun_M}_{\lambdach}}$ the restriction to the corresponding connected component of $\Bun_M$.

\medskip

Second, unlike above, consider now the map $Z^P_{\Bun_M} \to \Bun_M$ induced by the forgetful composite map
$$(\overline{G/U_P}) / M \times P^- \ \longto \ \cdot/ M \times P^- \ \longto \ \cdot/M \, .$$
Then we define the Zastava space $\tilde Z^P$ as the fiber of this map over the trivial $M$-bundle. The discussion of Subsections \ref{Zastava basic properties} and \ref{Zastava stratification} above applies to the Zastava space $\tilde Z^P$ as well, with the analogous notation.

\medskip

\sssec{Embeddings of Zastava spaces into affine Grassmannians}

\medskip

Next let $\Gr_G^{\check\theta}$ denote the Beilinson-Drinfeld affine Grassmannian parametrizing triples $(F_G, D, \eta)$ consisting of a $G$-bundle $F_G$ on the curve $X$, a $\Lambdach_{G,P}^{pos}$-valued divisor $D \in X^{\check\theta}$, and a trivialization of the $G$-bundle $F_G$ on the complement of the support of the divisor $D$. By construction the space $\Gr_G^{\check\theta}$ admits a forgetful map $\Gr_G^{\check\theta} \to X^{\check\theta}$, and is factorizable with respect to this map in the sense of Subsection \ref{Factorization of Gr} above. We recall from \cite{BFGM} that the Zastava spaces $Z^{P, \check\theta}$ and $\tilde Z^{P^-, \check\theta}$ admit natural locally closed embeddings into $\Gr_G^{\check\theta}$ which are compatible with the factorization structures.

\medskip

\sssec{Sections for Zastava spaces}
\label{Sections for Zastava spaces}
By Lemma \ref{bar M lemma} above, the inclusion $\overline{M} \into \overline{G/U_P}$ induces a section
$$\sigma_Z: \ \Gr_{M, G-pos}^{\check\theta} \ \longto \ Z^{P, \check\theta}$$
of the projection $Z^{P, \check\theta} \to \Gr_{M, G-pos}^{\check\theta}$. This section in fact maps $\Gr_{M, G-pos}^{\check\theta}$ isomorphically onto the stratum of maximal defect ${}_{\check\theta}Z^{P,\check\theta}$. Analogously we obtain a section
$$\sigma_{Z^-}: \ \Gr_{M, G-pos}^{\check\theta} \ \longto \ \tilde Z^{P^-, \check\theta}$$
of the projection map $\tilde Z^{P^-, \check\theta} \to \Gr_{M, G-pos}^{\check\theta}$.

\medskip

\sssec{Contractions for Zastava spaces}
\label{Contractions for Zastava spaces}
Next recall from \cite{MV} that any cocharacter $\check\lambda: \BG_m \to T$ naturally gives rise to an action of $\BG_m$ on the Beilinson-Drinfeld affine Grassmannian $\Gr_G^{\check\theta}$ which leaves the forgetful map $\Gr_G^{\check\theta} \to X^{\check\theta}$ invariant. Fix a cocharacter $\nu_M: \BG_m \to Z_M \subset T$ which contracts $U_{P^-}$ to the element $1 \in U_{P^-}$ when acting by conjugation.
Then it is shown in \cite{BFGM} that the corresponding $\nu_M$-action of $\BG_m$ on $\Gr_G^{\check\theta}$ preserves the subspace $Z^{P, \check\theta}$, and that this $\BG_m$-action contracts $Z^{P, \check\theta}$ onto the section $\sigma_Z$ above, i.e.: The action map of this $\BG_m$-action extends to a map
$$\BA^1 \times Z^{P, \check\theta} \ \longto \ Z^{P, \check\theta}$$
such that the composition
$$Z^{P, \check\theta} \ = \ \{ 0 \} \times Z^{P, \check\theta} \ \longinto \ \BA^1 \times Z^{P, \check\theta} \ \longto \ Z^{P, \check\theta}$$
agrees with the composition of the projection and the section
$$Z^{P, \check\theta} \ \longto \ \Gr_{M, G-pos}^{\check\theta} \ \stackrel{\sigma_Z}{\longto} \ Z^{P, \check\theta} \, .$$

\medskip

Analogously, the $(-\nu_M)$-action of $\BG_m$ on $\Gr_G^{\check\theta}$ preserves the subspace $\tilde Z^{P^-, \check\theta}$ and contracts $\tilde Z^{P^-, \check\theta}$ onto the section $\sigma_{Z^-}$.

\bigskip
\bigskip
\bigskip

\ssec{Stratification of the local models}

The stratification of the fiber $\VinBun_G|_{c_P}$ in Proposition \ref{defect stratification proposition} above induces an analogous stratification of $Y^{P, \check\theta}|_{c_P}$. To state it, let $\check\theta_1, \check\mu, \check\theta_2, \check\theta \in \Lambdach_{G,P}^{pos}$ with $\check\theta_1 + \check\mu + \check\theta_2 = \check\theta$. Then the strata map $f$ from Subsection \ref{Strata maps} above induces a locally closed immersion
$${}_0Z^{P^-, \check\theta_1}_{{\Bun_M}_{-\check\theta_2 - \check\mu}} \underset{\Bun_M}{\times} \CH^{-\check\theta_2 - \check\mu, -\check\theta_2}_{M, G-pos} \underset{\Bun_M}{\times} {}_0Z^{P, \check\theta_2} \ \ \ \longinto \ \ \ Y^{P, \check\theta}|_{c_P} \, .$$
We denote the corresponding locally closed substack of $Y^{P, \check\theta}|_{c_P}$ by ${}_{\check\theta_1, \check\mu, \check\theta_2}Y^{P, \check\theta}|_{c_P}$. We then have:

\medskip

\begin{corollary}
The locally closed substacks ${}_{\check\theta_1, \check\mu, \check\theta_2}Y^{P, \check\theta}|_{c_P}$ form a stratification of $Y^{P, \check\theta}|_{c_P}$, i.e.: On the level of $k$-points the space $Y^{P, \check\theta}|_{c_P}$ is equal to the disjoint union
$$Y^{P, \check\theta}|_{c_P} \ \ = \ \ \bigcup_{\check\theta_1 + \check\mu + \check\theta_2 \, = \, \check\theta} \ {}_{\check\theta_1, \check\mu, \check\theta_2}Y^{P, \check\theta}|_{c_P} \, .$$
\end{corollary}

\medskip

For notational simplicity the stratum ${}_{0, \check\theta, 0}Y^{P, \check\theta}|_{c_P}$ of maximal defect $\check\mu = \check\theta$ will also be denoted by ${}_{\check\theta}Y^{P, \check\theta}|_{c_P}$. By definition we have:
$${}_{\check\theta}Y^{P, \check\theta}|_{c_P} \ \ = \ \ \Gr_{M, G-pos}^{\check\theta}$$

\bigskip

\ssec{Section and contraction for the local models}
\label{Section and contraction for the local models}

\sssec{The canonical idempotent $e_P$ in the Vinberg semigroup}
Using the section $\Fs$ of the map $\Vin_G \to T_{adj}^+$ from Subsection \ref{The canonical section} above, we define
$$e_P \ := \ \Fs(c_P) \ \in \ \Vin_G|_{c_P} \, .$$
The element $e_P$ is an idempotent for the multiplication in $\Vin_G$, i.e., it satisfies $e_P \cdot e_P = e_P$. By definition of $\Vin_{G, \geq P, strict}$, multiplication by $e_P$ in $\Vin_G$ from the right or from the left defines a map
$$\Vin_{G, \geq P, strict} \ \longto \ \Vin_G|_{c_P} \, .$$
We will use the following fact (see e.g. \cite{W1}):

\begin{lemma}
\label{multiplying by the idempotent}
The image of the map
$$\Vin_{G, \geq P, strict} \ \longto \ \Vin_G|_{c_P}$$
obtained by multiplying by $e_P$ from the left agrees with the natural embedding
$$\overline{G/U_P} \ \ \longinto \ \ \Vin_G|_{c_P}$$
from Subsection \ref{Embedding of barM into Vin_G} above. Similarly, the image of the map
$$\Vin_{G, \geq P, strict} \ \longto \ \Vin_G|_{c_P}$$
obtained by multiplying by $e_P$ from the right agrees with the natural embedding
$$\overline{G/U_{P^-}} \ \ \longinto \ \ \Vin_G|_{c_P}$$
from Subsection \ref{Embedding of barM into Vin_G} above.
\end{lemma}

\medskip

\sssec{Embeddings for the local models}
\label{Embeddings for the local models}
Lemma \ref{multiplying by the idempotent} above gives rise to natural maps
$$\Vin_{G, \geq P, strict} / P \times U_{P^-} \ \ \longto \ \ (\overline{G/U_P}) / M \times U_{P^-}$$
and
$$\Vin_{G, \geq P, strict} / P \times U_{P^-} \ \ \longto \ \ (\overline{G/U_{P^-}}) / P \, .$$
Passing to mapping stacks we obtain natural maps
$$Y^{P, \check\theta} \ \longto \ Z^{P, \check\theta} \ \ \ \ \ \text{and} \ \ \ \ \ Y^{P, \check\theta} \ \longto \ \tilde Z^{P^-, \check\theta}$$
which are compatible with the projections to $\Gr_{M, G-pos}^{\check\theta}$.
Then the resulting map
$$\tau: \ \ Y^{P, \check\theta} \ \ \longto \ \ \tilde Z^{P^-, \check\theta} \underset{\Gr_{M, G-pos}^{\check\theta}}{\times} Z^{P, \check\theta} \ \times \ T_{adj, \geq P, strict}^+$$
obtained by taking the product of the above two maps to the Zastava spaces and the natural map $Y^{P, \check\theta} \to T_{adj, \geq P, strict}^+$ is a closed immersion.

\sssec{The section $\sigma$}
Part (c) of Lemma \ref{bar M lemma} and Lemma \ref{GIT quotient of strict locus} together imply that the inclusion
$\overline{M} \into \Vin_G|_{c_P}$
from Subsection \ref{Embedding of barM into Vin_G} above induces a section
$$\sigma: \ \ \Gr_{M, G-pos}^{\check\theta} \ \ \longinto \ \ Y^{P, \check\theta}$$
of the projection map
$$\pi: \ Y^{P, \check\theta} \ \longto \ \Gr_{M, G-pos}^{\check\theta} \, .$$
By construction this section maps $\Gr_{M, G-pos}^{\check\theta}$ isomorphically onto the stratum of maximal defect:
$$\sigma: \ \Gr_{M, G-pos}^{\check\theta} \ \ \stackrel{\cong}{\longto} \ \ {}_{\check\theta}Y^{P, \check\theta}|_{c_P} \ \ \longinto \ \ Y^{P, \check\theta}$$

\medskip

Alternatively, the section $\sigma$ can be constructed as follows: Define a map
$$\Gr_{M, G-pos}^{\check\theta} \ \ \longto \ \ \tilde Z^{P^-, \check\theta} \underset{\Gr_{M, G-pos}^{\check\theta}}{\times} Z^{P, \check\theta} \ \times \ T_{adj, \geq P, strict}^+$$
by choosing the map to the first factor to be the section $\sigma_{Z^-}$ from Subsection \ref{Sections for Zastava spaces} above, the map to the second factor to be the section $\sigma_Z$, and the map to the third factor to be the constant map with value $c_P \in T_{adj, \geq P, strict}^+$. Then by construction of the embedding $\tau$ from Subsection \ref{Embeddings for the local models} above, this map factors through the subspace $Y^{P, \check\theta}$ and agrees with the section $\sigma$.

\medskip

\sssec{Contracting the local model onto the section}
\label{Contracting the local model onto the section}
As in Subsection \ref{Contractions for Zastava spaces} above we fix a cocharacter $\check\nu_M: \BG_m \to Z_M \subset T$ which contracts $U_{P^-}$ to the element $1 \in U_{P^-}$, and consider the corresponding $\nu_M$-action on $Z^{P, \check\theta}$ and the corresponding $(-\check\nu_M)$-action on $\tilde Z^{P^-, \check\theta}$. Then we let $\BG_m$ act on the product
$$\tilde Z^{P^-, \check\theta} \underset{\Gr_{M, G-pos}^{\check\theta}}{\times} Z^{P, \check\theta} \ \times \ T_{adj, \geq P, strict}^+$$
as follows: We act on the first factor via the $(-\check\nu_M)$-action on $\tilde Z^{P^-, \check\theta}$, we act on the second factor via the $\nu_M$-action on $Z^{P, \check\theta}$, and we act on the third factor via the cocharacter $(-2 \check\nu_M): \BG_m \to T$ and the usual action of $T$ on $T_{adj, \geq P, strict}^+$. This $\BG_m$-action preserves the subspace $Y^{P, \check\theta}$, and in fact we have:

\medskip

\begin{lemma}
This $\BG_m$-action contracts the subspace $Y^{P, \check\theta}$ onto the section~$\sigma$, i.e.: The action map of this $\BG_m$-action extends to a map
$$\BA^1 \times Y^{P, \check\theta} \ \longto \ Y^{P, \check\theta}$$
such that the composition
$$Y^{P, \check\theta} \ = \ \{ 0 \} \times Y^{P, \check\theta} \ \longinto \ \BA^1 \times Y^{P, \check\theta} \ \longto \ Y^{P, \check\theta}$$
agrees with the composition of the projection and the section
$$Y^{P, \check\theta} \ \stackrel{\pi}{\longto} \ \Gr_{M, G-pos}^{\check\theta} \ \stackrel{\sigma}{\longto} \ Y^{P, \check\theta} \, .$$
\end{lemma}

\medskip

\begin{proof}
Since the map $\tau$ from Subsection \ref{Embeddings for the local models} above is a closed immersion and is compatible with the projection maps to $\Gr_{M, G-pos}^{\check\theta}$, it suffices to show that the action contracts the ambient space
$$\tilde Z^{P^-, \check\theta} \underset{\Gr_{M, G-pos}^{\check\theta}}{\times} Z^{P, \check\theta} \ \times \ T_{adj, \geq P, strict}^+$$
onto the section $\sigma$. By Subsection \ref{Contractions for Zastava spaces} above we only have to show that the action of $\BG_m$ on $T_{adj}^+$ via the composition $\BG_m \stackrel{-2 \check\nu_M}{\longto} T \longto T_{adj}^+$ contracts $T_{adj, \geq P, strict}^+$ onto the point $c_P \in T_{adj, \geq P, strict}^+$. To see this, let $i \in \CI$. Then for $i \in \CI_M$ the integer $< - 2 \check\nu_M, \alpha_i>$ is equal to $0$ since $\check\nu_M$ factors through the center of $M$; if $i \notin \CI_M$ the integer $<-2\check\nu_M, \alpha_i>$ is positive since $\check\nu_M$ contracts $U_{P^-}$, as desired.
\end{proof}

\bigskip
\bigskip
\bigskip

\section{Proofs II --- Sheaves}
\label{Proofs II --- Sheaves}

\ssec{Restatement of geometric theorems for the local models}

By the exact same argument as in \cite{BFGM}, \cite{BG2}, or \cite{Sch1}, it suffices to prove the theorems stated in Section \ref{Statements of theorems -- Geometry} above on the level of the local models. For the convenience of the reader, we now restate the theorems in the notation of the local models:

\medskip

\sssec{Nearby cycles theorem}
\label{Nearby cycles theorem}
As above we fix a parabolic $P$ of $G$ and consider the line $L_P = \BA^1 \into T_{adj}^+ = \BA^r$
passing through the points $c_G$ and $c_P$ of $T_{adj}^+ = \BA^r$, identifying the point $1 \in \BA^1$ with the point $c_G$ and the point $0 \in \BA^1$ with the point $c_P$. We denote by $Y^{P, \check\theta}|_{L_P}$ the restriction of the family $Y^{P, \check\theta} \to T_{adj, \geq P, strict}^+$ to the line $L_P = \BA^1$ and by $\Psi_P \in D(Y^{P, \check\theta}|_{c_P})$ the corresponding nearby cycles of the IC-sheaf
$$\IC_{Y^{P, \check\theta}|_{L_P \setminus \{ 0 \} }} \ \ = \ \ \Qellbar[\dim Y^{P, \check\theta}|_{L_P \setminus \{ 0 \} }](\tfrac{1}{2} \dim Y^{P, \check\theta}|_{L_P \setminus \{ 0 \} } \, .$$
of its $G$-locus
Next recall from Subsection \ref{The G-positive affine Grassmannian for M} above that for any $\check\theta \in \Lambdach_{G,P}^{pos}$ the fiber of the forgetful map
$$\CH^{-\check\theta, 0}_{M, G-pos} \ \longto \ \Bun_{M,0}$$
over the trivial bundle in $\Bun_{M,0}$ naturally identifies with $\Gr_{M, G-pos}^{\check\theta}$. We will denote the corresponding version of the complex $\widetilde \Omega_P$ on the space $\Gr_{M, G-pos}^{\check\theta}$ from Subsection \ref{The complex tilde-Omega_P} above by $\widetilde\Omega_P^{\check\theta}$; i.e., we define the complex $\widetilde\Omega_P^{\check\theta}$ on $\Gr_{M, G-pos}^{\check\theta}$ as the pushforward
$$\widetilde \Omega_P^{\check\theta} \ \ := \ \ \pi_{Z,!} \, \bigl( \IC_{{}_0Z^{P, \check\theta}} \bigr) \, .$$
Then to prove Theorem \ref{nearby cycles theorem} above we have to show:

\medskip

\begin{theorem}
\label{nearby cycles theorem for local models}
The $*$-restriction of $\Psi_P$ to the stratum of maximal defect
$$\Gr_{M, G-pos}^{\check\theta} \ \ = \ \ {}_{\check\theta}Y^{P, \check\theta}|_{c_P} \ \ \ \longinto \ \ \ Y^{P, \check\theta}|_{c_P}$$
of $Y^{P, \check\theta}|_{c_P}$ is equal to the complex $\BD \, \widetilde\Omega_P^{\check\theta}$.
\end{theorem}

\bigskip

\sssec{The $*$-extension of the constant sheaf}
\label{*-extension of the constant sheaf for local models}
We begin with a basic lemma needed to reduce the proof of Theorem \ref{ij} to a version for the local models. To state it, let $\VinBun_{G, \geq P, strict}$ denote the restriction of the family $\VinBun_G \to T_{adj}^+$ to the closed subvariety
$$T_{adj, \geq P}^+ \ \longinto \ T_{adj}^+ = \BA^r \, .$$
We denote by $\VinBun_{G, \geq P, strict, G}$ the $G$-locus of this family, i.e., the restriction of this family to the open stratum $T_{adj, \geq P, strict, G}^+$ of $T_{adj, \geq P, strict}^+$ introduced in Subsection \ref{Strict P-loci} above. Let
$$j_{\geq P, G, strict}: \ \ \VinBun_{G, \geq P, strict, G} \ \longinto \ \VinBun_{G, \geq P, strict}$$
denote the corresponding open inclusion, and as before let $j_G$ denote the open inclusion
$$j_G: \ \VinBun_{G,G} \ \longinto \ \VinBun_G \, .$$
For the purpose of stating the lemma we denote by $i_{P, \lambdach_1, \lambdach_2}$ the inclusion of the stratum
$$i_{P, \lambdach_1, \lambdach_2}: \ \ \sideset{_{\lambdach_1, \lambdach_2}}{_G}\VinBun|_{c_P} \ \ \longinto \ \ \VinBun_G|_{c_P} \, .$$
Then we have:

\medskip

\begin{lemma}
\label{lemma for *-extension for local models}
$$i_{P, \lambdach_1, \lambdach_2}^* \, j_{G,*} \, (\Qellbar)_{\VinBun_{G,G}} \ \ = \ \ i_{P, \lambdach_1, \lambdach_2}^* \, j_{\geq P, G, strict, *} \, (\Qellbar)_{\VinBun_{G, \geq P, strict, G}}$$
\end{lemma}

\medskip

\begin{proof}
Using analogous notation to above, consider the open inclusion
$$j_{\geq P, G}: \ \ \VinBun_{G, G} \ \longinto \ \VinBun_{G, \geq P} \, .$$
Then since $\VinBun_{G, \geq P}$ is an open substack of $\VinBun_G$ containing both the open substack $\VinBun_{G,G}$ and the locus $\VinBun_{G,P}$ we have:
$$i_{P, \lambdach_1, \lambdach_2}^* \, j_{G,*} \, (\Qellbar)_{\VinBun_{G,G}} \ \ = \ \ i_{P, \lambdach_1, \lambdach_2}^* \, j_{\geq P, G, *} \, (\Qellbar)_{\VinBun_{G, G}}$$
Next observe that $T_{adj, \geq P}^+$ by definition splits as a product
$$T_{adj, \geq P}^+ \ \ =  \ \ T_{adj, \geq P, strict}^+ \ \times \prod_{i \in \CI_M} \BA^1 \setminus \{ 0 \} \, ,$$
and consider the $T_{adj}$-action on $\VinBun_G$ from Subsection \ref{The T_adj-action on VinBun_G} above lifting the $T_{adj}$-action on $T_{adj}^+$. Then the subgroup
$$\prod_{i \in \CI_M} \BG_m \ \ \longinto \ \ \prod_{i \in \CI} \BG_m \ = \ T_{adj}$$
acts simply transitively on the second factor in the above product decomposition. Lifting the action of this subgroup to $\VinBun_{G, \geq P}$ then yields a product decomposition
$$\VinBun_{G, \geq P} \ \ = \ \ \VinBun_{G, \geq P, strict} \ \times \prod_{i \in \CI_M} \BA^1 \setminus \{ 0 \} \, .$$
This product decomposition identifies the open substack $\VinBun_{G,G}$ of the left hand side with the open substack
$$\VinBun_{G, \geq P, strict, G} \ \times \prod_{i \in \CI_M} \BA^1 \setminus \{ 0 \}$$
of the right hand side, and the fiber $\VinBun_G|_{c_P}$ with the closed substack
$$\VinBun_G|_{c_P} \times \prod_{i \in \CI_M} \{ 1 \} \, ;$$
this implies the claim.
\end{proof}

\medskip

By Lemma \ref{lemma for *-extension for local models} above the assertion of Theorem \ref{ij} now reduces to the following analog for the local models:

\medskip

\begin{theorem}
\label{ij for the local models}
The $*$-restriction of the $*$-extension $j_{G,*} \, \IC_{Y^{P, \check\theta}_G}$ to the stratum of maximal defect
$$\Gr_{M, G-pos}^{\check\theta} \ = \ {}_{\check\theta}Y^{P, \check\theta}|_{c_P} \ \longinto \ Y^{P, \check\theta}$$
of the fiber $\VinBun_G|_{c_P}$ is equal to the complex
$$\BD \, \widetilde\Omega_P^{\check\theta} \, \otimes \, H^*(T_{adj, \geq P, strict, G}^+, \Qellbar) [r-r_M](\tfrac{r-r_M}{2}) \ \ =$$
$$= \ \ \BD \, \widetilde\Omega_P^{\check\theta} \, \otimes \, H^*((\BA^1 \setminus \{0\})^{r-r_M}, \Qellbar) [r-r_M](\tfrac{r-r_M}{2}) \, .$$
\end{theorem}

\bigskip

\ssec{Proof of Theorem \ref{*-extension of the constant sheaf for local models}}

In Subsection \ref{Section and contraction for the local models} above we have constructed a $\BG_m$-action which contracts the local model $Y^{P, \check\theta}$ onto the section
$$\sigma: \ \Gr_{M, G-pos}^{\check\theta} \ \ \stackrel{\cong}{\longto} \ \ {}_{\check\theta}Y^{P, \check\theta}|_{c_P} \ \ \longinto \ \ Y^{P, \check\theta}$$
of the projection map
$$\pi: \ Y^{P, \check\theta} \ \longto \ \Gr_{M, G-pos}^{\check\theta} \, .$$
In this setting, the well-known \textit{contraction principle} (see for example \cite[Sec. 3]{Br} or \cite[Sec. 5]{BFGM}) for contracting $\BG_m$-actions states:

\medskip

\begin{lemma}
\label{contraction principle}
For any $\BG_m$-monodromic object $F \in D(Y^{P, \check\theta})$ there exists a natural isomorphism
$$\sigma^* F \ \cong \ \pi_* F \, .$$
\end{lemma}

\medskip

We can now complete the proof of Theorem \ref{*-extension of the constant sheaf for local models}:

\medskip

\begin{proof}[Proof of Theorem \ref{*-extension of the constant sheaf for local models}]
Denote by
$$\pi_G: \ Y^{P, \check\theta}_G \ \longto \ \Gr_{M, G-pos}^{\check\theta}$$
the restriction of the projection $\pi$ to the $G$-locus $Y^{P, \check\theta}_G$ of $Y^{P, \check\theta}$. Applying the contraction principle to the $\BG_m$-equivariant sheaf $j_{G,*} \, \IC_{Y^{P, \check\theta}_G}$, we compute the desired restriction as
$$\sigma^* \, j_{G,*} \, \IC_{Y^{P, \check\theta}_G} \ \ = \ \ \pi_* \, j_{G,*} \, \IC_{Y^{P, \check\theta}_G} \ \ = \ \ \pi_{G,*} \, \IC_{Y^{P, \check\theta}_G} , .$$
But Remark \ref{triviality of fiber bundles for local models} and Remark \ref{fiber of local model over 1} show that
$$Y^{P, \check\theta}_G \ \ \cong \ \ {}_0Z^{P, \check\theta} \ \times \ T_{adj, \geq P, strict, G}^+$$
as spaces over $T_{adj, \geq P, strict, G}^+$. Thus we conclude that the desired restriction is equal to
$$\pi_{G,*} \, \IC_{Y^{P, \check\theta}_G} \ \ = \ \ \pi_{G,*} \, (\IC_{{}_0Z^{P, \check\theta}} \ \boxtimes \ T_{adj, \geq P, strict, G}^+) \ \ =$$
$$= \ \ \pi_{Z,*} \, \IC_{{}_0Z^{P, \check\theta}} \ \otimes \ H^*(T_{adj, \geq P, strict, G}^+, \IC_{T_{adj, \geq P, strict, G}^+}) \ \ = $$
$$= \ \ \BD \, \widetilde\Omega_P^{\check\theta} \, \otimes \, H^*(T_{adj, \geq P, strict, G}^+, \Qellbar) [r-r_M](\tfrac{r-r_M}{2}) \, .$$
\end{proof}

\bigskip

\ssec{Proof of Theorem \ref{nearby cycles theorem for local models}}

\begin{proof}[Proof of Theorem \ref{nearby cycles theorem for local models}]

For the purpose of the proof we denote the inclusion of the special fiber $Y^{P, \check\theta}|_{c_P}$ of the one-parameter family
$$Y^{P, \check\theta}|_{L_P} \ \ \longto \ \ L_P \ = \ \BA^1$$
by
$$i_P: \ Y^{P, \check\theta}|_{c_P} \ \longinto \ Y^{P, \check\theta}|_{L_P}$$
the inclusion of its $G$-locus by
$$j_G: \ Y^{P, \check\theta}|_{L_P \setminus \{ 0 \} } \ \longinto \ Y^{P, \check\theta}|_{L_P} \, ,$$
and the projection map of its $G$-locus by
$$\pi_G: \ Y^{P, \check\theta}|_{L_P \setminus \{ 0 \} } \ \longto \ \Gr_{M, G-pos}^{\check\theta} \, .$$
Recall furthermore that
$$Y^{P, \check\theta}|_{L_P \setminus \{ 0 \} } \ \ \ \cong \ \ \ {}_0Z^{P, \check\theta} \ \times \ (L_P \setminus \{ 0 \})$$
as spaces over $L_P \setminus \{ 0 \}$.
Finallly, observe that the $\BG_m$-action on $Y^{P, \check\theta}$ from Subsection \ref{Contracting the local model onto the section} above preserves the one-parameter family $Y^{P, \check\theta}|_{L_P}$ and hence contracts it onto the stratum of maximal defect
$$\Gr_{M, G-pos}^{\check\theta} \ \ = \ \ {}_{\check\theta}Y^{P, \check\theta}|_{c_P} \ \ \ \longinto \ \ \ Y^{P, \check\theta}|_{c_P} \, ;$$
we will hence be able to apply the contraction principle stated in Lemma \ref{contraction principle} above. To compute $\sigma^* \Psi_P$ we first apply Koszul duality for nearby cycles and find
$$\sigma^* \, \Psi_P \ \ = \ \ \sigma^* \, i_P^* \, j_{G,*} \, \IC_{Y^{P, \check\theta}|_{L_P \setminus \{ 0 \} }}[-1](-\tfrac{1}{2}) \ \underset{H^*(L_P \setminus \{ 0 \}, \Qellbar)}{\otimes} \ \Qellbar \ \ =$$
$$= \ \ \sigma^* \, j_{G,*} \, \IC_{Y^{P, \check\theta}|_{L_P \setminus \{ 0 \} }}[-1](-\tfrac{1}{2}) \ \underset{H^*(L_P \setminus \{ 0 \}, \Qellbar)}{\otimes} \ \Qellbar \ .$$
Applying the contraction principle from Lemma \ref{contraction principle} above to the $\BG_m$-equivariant sheaf $j_{G,*} \, \IC_{Y^{P, \check\theta}|_{L_P \setminus \{ 0 \} }}$ we therefore compute
$$\sigma^* \, \Psi_P \ \ = \ \ \pi_* \, j_{G,*} \, \IC_{Y^{P, \check\theta}|_{L_P \setminus \{ 0 \} }}[-1](-\tfrac{1}{2}) \ \underset{H^*(L_P \setminus \{ 0 \}, \Qellbar)}{\otimes} \ \Qellbar \ \ =$$
$$= \ \ \pi_{G,*} \, (\IC_{{}_0Z^{P, \check\theta}} \, \boxtimes \, (\Qellbar)_{L_P \setminus \{ 0 \}}) \ \underset{H^*(L_P \setminus \{ 0 \}, \Qellbar)}{\otimes} \ \Qellbar \ \ =$$
$$= \ \ \pi_{Z,*} \IC_{{}_0Z^{P, \check\theta}} \ \otimes \ H^*(L_P \setminus \{ 0 \}, \Qellbar) \ \underset{H^*(L_P \setminus \{ 0 \}, \Qellbar)}{\otimes} \ \Qellbar \ \ =$$
$$= \ \ \BD \, \widetilde\Omega_P^{\check\theta} \, ,$$
completing the proof.
\end{proof}

\bigskip
\bigskip
\bigskip

\section{Proofs III -- Bernstein asymptotics}
\label{Proofs III -- Bernstein asymptotics}

\bigskip

We first recall two well-known facts about the nearby cycles functor; see Subsection~\ref{Recollections notation} above for our conventions and normalizations regarding nearby cycles.

\medskip

\sssec{Nearby cycles and fiber products}

Next let $Y \to \BA^1$ and $Y' \to \BA^1$ be two stacks or schemes over $\BA^1$, let $F$ and $F'$ be objects of $\D\bigl( Y|_{\BA^1 \setminus \{0\}}\bigr)$ and $\D\bigl(Y'|_{\BA^1 \setminus \{0\}}\bigl)$, and denote
$$F \underset{\ \BA^1}{\boxtimes} F' \ := \ \bigl( F \boxtimes F' \bigr) \big|^*_{Y \underset{ \, \BA^1}{\times} Y'}[-1](-\tfrac{1}{2}) \, .$$
Recall that we denote the unipotent nearby cycles functor by $\Psi$; denote by $\Psi_{full}$ the full nearby cycles functor. Then we have the following lemma (see \cite[Sec. 5]{BB}), which we will in fact only apply in the case where $\Psi_{full} = \Psi$:

\medskip

\begin{lemma}
\label{Psi and products}
On the product $Y|_{\{0\}} \times Y'|_{\{0\}}$ there exists a canonical isomorphism
$$\Psi_{full}(F \underset{\ \BA^1}{\boxtimes} F') \ \ = \ \ \Psi_{full}(F) \boxtimes \Psi_{full}(F') \, .$$
\end{lemma}

\medskip

\sssec{Unipotence in the equivariant setting}

Next we recall the following fact regarding the unipotence of the nearby cycles (see e.g. \cite[Lemma 11]{G2}):

\medskip

\begin{lemma}
\label{unipotence lemma}
Let $Y \to \BA^1$ be a scheme or a stack over $\BA^1$, and assume that there exists a $\BG_m$-action on $Y$ which lifts the standard $\BG_m$-action on~$\BA^1$. Let $F$ be a $\BG_m$-equivariant perverse sheaf on $Y|_{\BA^1 \setminus \{0\}}$. Then the full nearby cycles of $F$ are automatically unipotent, i.e., we have $\Psi_{full}(F) = \Psi(F)$.
\end{lemma}

\medskip

\ssec{Factorization of nearby cycles}

\begin{proof}[Proof of Proposition \ref{factorization of nearby cycles proposition}]
It suffices to prove the claim on the level of the local models $Y^B$. By Lemma \ref{unipotence lemma} the full nearby cycles functor of the principal degeneration $\VinBun_G^{princ}$ is unipotent. We can thus apply Lemma \ref{Psi and products} above for the unipotent nearby cycles $\Psi^{princ}$; the claim then follows from the factorization in families for the local models $Y^B$ stated in Lemma \ref{factorization in families} above.
\end{proof}

\bigskip

\ssec{Proof of Theorem \ref{function corresponding to Psi}}

\sssec{The complexes $\Omega^{\check\theta}$ and $U^{\check\theta}$ from \cite{BG2}}
Let $\check\theta \in \Lambdach_G^{pos}$. In \cite{BG1}, \cite{BG2}, and \cite{BFGM} certain complexes $\Omega^{\check\theta}$ and $U^{\check\theta}$ on $X^{\check\theta}$ are introduced; we refer to these articles for their definitions and motivation. Here we will only be interested in these complexes on the level of the Grothendieck group, and in a description of the complex $\widetilde \Omega_B^{\check\theta}$ in the Grothendieck group in terms of $\Omega^{\check\theta}$ and $U^{\check\theta}$ which will allow us to compute the function corresponding to $\widetilde \Omega_B^{\check\theta}$ under the sheaf-function correspondence.

\medskip

To state the descriptions of $\Omega^{\check\theta}$ and $U^{\check\theta}$, we introduce the following notation. First, for $\check\theta_1, \check\theta_2 \in \Lambdach_G^{pos}$ we denote by
$$\add: \ X^{\check\theta_1} \times X^{\check\theta_2} \ \longto \ X^{\check\theta_1 + \check\theta_2}$$
the addition map of $\Lambdach_G^{pos}$-valued effective divisors on $X$.
Next, to any Kostant partition
$$\CK: \ \ \thetacheck \ = \ \sum_{\betacheck \in \check R^+} n_{\betacheck} \betacheck$$
of a positive coweight $\thetacheck \in \Lambdach_G^{pos}$ we associate the partially symmetrized power
$$X^{\CK} \ \ := \ \ \prod_{\betacheck \in \check R^+} X^{(n_{\betacheck})}$$
of the curve $X$. We denote by
$$i_{\CK}: \ \ X^{\CK} \ \longto \ X^{\thetacheck}$$
the finite map defined by adding $\Lambdach_G^{pos}$-valued divisors.
Finally, for a local system $L$ on $X$ we denote by $\Lambda^{(n)}(L)$ the $n$-th external exterior power of $L$ on $X^{(n)}$.
We can now state the following result from \cite[Section 3.3]{BG2}:

\medskip

\begin{lemma}
\label{Omega lemma}
In the Grothendieck group on $X^{\check\theta}$ we have:
$$\Omega^{\check\theta} \ \ = \ \ \bigoplus_{\CK \, \in \, \Kost(\check\theta)} i_{\CK,*} \, \bigl( \, \underset{\check\beta}{\boxtimes} \, \Lambda^{(n_{\check\beta})}(\Qellbar_X) \, [n_{\check\beta}](n_{\check\beta})\bigr)$$
\end{lemma}

\bigskip

Similarly we recall from \cite[Theorem 4.5]{BFGM}:

\medskip

\begin{lemma}
\label{U lemma}
In the Grothendieck group on $X^{\check\theta}$ we have:
$$U^{\check\theta} \ \ = \ \ \bigoplus_{\CK \, \in \, \Kost(\check\theta)} i_{\CK,*} \, \Qellbar_{X^{\CK}}[0](0)$$
\end{lemma}

\medskip

\sssec{Description of $\widetilde \Omega_B^{\check\theta}$ in the Grothendieck group}
We can now recall the aforementioned description of the complex $\widetilde \Omega_B^{\check\theta}$ on $X^{\check\theta}$ in the Grothendieck group, which follows directly from Corollary 4.5 of \cite{BG2}:

\begin{lemma}
\label{tilde Omega lemma}
In the Grothendieck group on $X^{\check\theta}$ we have:
$$\widetilde \Omega_B^{\check\theta} \ \ = \ \ \sum_{\check\theta_1 + \check\theta_2 = \check\theta} add_* \Bigl( \Omega^{\check\theta_1} \; \boxtimes \ U^{\check\theta_2} \Bigr)$$
Here the sum runs over all pairs of positive coweights $(\check\theta_1, \check\theta_2)$ satisfying $\check\theta_1 + \check\theta_2 = \check\theta$.
\end{lemma}

\medskip

\sssec{Proof of Theorem \ref{function corresponding to Psi}}

\begin{proof}[Proof of Theorem \ref{function corresponding to Psi}]
Given a Kostant partition $\CK: \check\theta = \sum_{\betacheck \in \check R^+} n_{\betacheck} \betacheck$ we define $|\CK| := \sum_{\betacheck \in \check R^+} n_{\betacheck}$; in other words, we define $|\CK| := \dim X^{\CK}$. Let now $x \in X(\BF_q)$. Then Lemmas \ref{Omega lemma}, \ref{U lemma}, and \ref{tilde Omega lemma} above together imply that in the Grothendieck group we have
$$(\BD \widetilde \Omega_B^{\check\theta})|^*_{\check\theta x} \ = \ \sum_{\check\theta_1 + \check\theta_2 = \check\theta} \, \sum_{\CK_1, \CK_2} \, \Qellbar \, [2 \, |\CK_1|] \, (|\CK_1|) \ \otimes \ \Qellbar \, [\CK_2] \, (0)$$
where $\CK_1$ ranges over the set $\Kost(\check\theta_1)$, where $\CK_2$ ranges over the set $\Kost(\check\theta_2)$, and where $\CK_2$ is \textit{simple} in the sense that each integer $n_{\check\beta}$ appearing in $\CK_2$ is either $0$ or $1$. Indeed, if one of the integers $n_{\check\beta}$ is larger than $1$, the corresponding stalk of the external exterior power $\Lambda^{(n_{\check\beta})}(\Qellbar_X)$ vanishes. To reformulate the last formula, note that giving a sum decomposition $\check\theta_1 + \check\theta_2 = \check\theta$ and Kostant partitions $\CK_1$ and $\CK_2$ as above with $\CK_2$ simple is equivalent to giving a Kostant partition $\CK$ of $\check\theta$ together with a subset $S$ of the set of roots appearing in $\CK$. Thus the above formula can be rewritten as
$$(\BD \widetilde \Omega_B^{\check\theta})|^*_{\check\theta x} \ = \ \sum_{\CK} \, \sum_{S \subset R_{\CK}} \, \Qellbar \, [2 \, |\CK| - |S|] \, (|\CK| - |S|)$$
where $\CK$ ranges over the set $\Kost(\check\theta)$, where $S$ ranges over all subsets of $R_{\CK}$, and where $|S|$ denotes the cardinality of $S$. Passing to the value of the corresponding function, we find
$$\sum_{\CK} \, \sum_{S \subset R_{\CK}} \, (-1)^{|S|} \, q^{|S|} \, q^{-|K|} \, .$$
But since
$$\sum_{S \subset R_{\CK}} \, (-1)^{|S|} \, q^{|S|} \ \ = \ \ (1-q)^{R_{\CK}}$$
the last expression is equal to
$$\sum_{\CK \, \in \, \Kost(\check \theta)} (1-q)^{|R_{\CK}|} \; q^{-|\CK|} \, .$$
Observing that the codimension of the stratum of defect $\check\theta$ is $\langle 2 \rho, \check\theta \rangle$ and recalling our normalization of IC-sheaves to be pure of weight $0$, the formula follows as stated.
\end{proof}

\bigskip
\medskip

\sssec{Proof of Theorem \ref{theorem Sakellaridis's conjecture}}
\label{Proof of Sakellaridis's conjecture}

\begin{proof}[Proof of Theorem \ref{theorem Sakellaridis's conjecture}]
We first recall Sakellaridis's Gindikin-Karpelevich-formula for $\Asymp(\phi_0)$, using the same notation as in \cite{Sak1}, \cite{Sak2}. In particular we use the notation
$$e^{\check\lambda} \ := \ q^{\langle \rho, \check\lambda \rangle} \, \mathbbold{1}_{\check\lambda}$$
for a coweight $\lambdach \in \Lambdach_G$. Then Sakellaridis's formula from \cite[Section 6]{Sak1} states:
$$\Asymp(\phi_0) \ = \ \prod_{\check\alpha \in \check R^+} \frac{1 - e^{\check\alpha}}{1 - q^{-1} e^{\check\alpha}}$$
Expanding the denominators as geometric series, the formula becomes
$$\prod_{\check\alpha \in \check R^+} \Bigl( e^{0 \cdot \check\alpha} + \sum_{i_{\check\alpha}=1}^{\infty} q^{-i_{\check\alpha}} (1-q) e^{i_{\check\alpha} \check\alpha} \Bigr) \, .$$
Multiplying out we obtain the expression
$$\sum_{\CK: \ \sum_{\check\alpha} i_{\check\alpha} \check\alpha} \ (1-q)^{\sum_{\check\alpha: \, i_{\check\alpha} \neq 0} 1} \cdot q^{- \sum_{\check\alpha} i_{\check\alpha}} \ e^{\sum_{\check\alpha} i_{\check\alpha} \check\alpha}$$
where the sum is running over all Kostant partitions $\CK: \sum_{\check\alpha} i_{\check\alpha} \check\alpha$ of all positive coweights of $G$. Since $\sum_{\check\alpha: \, i_{\check\alpha} \neq 0} 1 = |R_{\CK}|$ and since $\sum_{\check\alpha} i_{\check\alpha} = |\CK|$ the last expression in turn is equal to
$$\sum_{\check\theta \in \Lambdach_G^{pos}} \sum_{\CK \in \Kost(\check\theta)} (1-q)^{|R_{\CK}|} \, q^{-|\CK|} \, q^{\langle \rho, \check\theta \rangle} \, \mathbbold{1}_{\check\theta} \, ,$$
as desired.
\end{proof}

\end{document}